\documentclass[11pt,a4paper]{amsart}

\usepackage{graphicx}
\graphicspath{{./Figures/}}
\usepackage{subfigure}
\usepackage{caption}
\usepackage[inline]{enumitem}
\usepackage{amssymb}
\usepackage{xcolor}
\usepackage{stmaryrd}
\usepackage{appendix}
\usepackage{multirow}
\usepackage{booktabs}
\usepackage{accents}
\usepackage{orcidlink}

\theoremstyle{plain}
\newtheorem{theorem}{Theorem}[section]
\newtheorem{lemma}[theorem]{Lemma}

\newtheorem{proposition}[theorem]{Proposition}

\theoremstyle{definition}
\newtheorem{definition}[theorem]{Definition}

\theoremstyle{remark}
\newtheorem{remark}[theorem]{Remark}

\numberwithin{equation}{section}

\newcommand{\leqs}{\leqslant}
\newcommand{\geqs}{\geqslant}
\def\bold#1{\mbox{\boldmath $#1$}}
\newcommand{\uu}[1]{\bold{#1}}

\newcommand{\abs}[1]{\lvert#1\rvert}
\newcommand{\D}{\partial}
\newcommand{\dd}{\mathrm{d}}
\newcommand{\dive}{\mathrm{div}\,}
\newcommand{\bdive}{\uu{\mathrm{div}}\,}
\newcommand{\divM}{\mathrm{div}_{\mathcal{M}}}

\newcommand{\Dt}{\partial_t}

\newcommand{\gm}{\gamma}

\newcommand{\grd}{\nabla}
\newcommand{\bgrd}{\uu{\nabla}}

\newcommand{\mbb}{\mathbb}

\newcommand{\mcal}{\mathcal}

\newcommand{\norm}[1]{\lVert#1\rVert}
\newcommand{\Norm}[1]{{\left\vert\kern-0.25ex\left\vert\kern-0.25ex\left\vert #1 
    \right\vert\kern-0.25ex\right\vert\kern-0.25ex\right\vert}}

\newcommand{\sign}{\mathrm{sign}}

\newcommand{\half}{\frac{1}{2}}
\newcommand{\veps}{\varepsilon}
\newcommand{\eps}{\epsilon}

\newcommand{\s}{\sigma}

\newcommand{\M}{\mcal{M}}
\newcommand{\E}{\mcal{E}}

\newcommand{\Ds}{D_\sigma}

\newcommand{\Eint}{\mcal{E}_{\mathrm{int}}}
\newcommand{\Ek}{\mcal{E}(K)}
\newcommand{\Lm}{L_{\M}(\Omega)}

\newcommand{\dt}{\delta t}

\newcommand{\intr}{\mathrm{int}}
\newcommand{\extr}{\mathrm{ext}}
\newcommand{\absq}[1]{\abs{#1}^2}
\def\bold#1{\mbox{\boldmath $#1$}}

\newcommand{\diam}{\text{diam}}

\newcommand\ubar[1]{%
\underaccent{\bar}{#1}}

\makeindex             

\begin{document}

\title[A Structure-Preserving Scheme]{A Structure-Preserving Scheme for the Euler System with Potential Temperature Transport}

\author[Arun]{K.\ R.\ Arun\scalebox{1.2}{\orcidlink{0000-0001-9676-861X}}}
\address{School of Mathematics, Indian Institute of Science Education and Research Thiruvananthapuram, Thiruvananthapuram 695551, India}  
\email{arun@iisertvm.ac.in, rahuldev19@iisertvm.ac.in}

\author[Ghorai]{R.\ Ghorai\,\scalebox{1.2}{\orcidlink{0000-0003-0549-3417}}}
\thanks{R.G.\ would like to thank the Ministry of Education, Government of India, for the PMRF fellowship support.} 

\date{\today}



\begin{abstract}
    We consider the compressible Euler equations with potential temperature transport, a system widely used in atmospheric modelling to describe adiabatic, inviscid flows. In the low Mach number regime, the equations become stiff and pose significant numerical challenges. We develop an all-speed, semi-implicit finite volume scheme that is asymptotic preserving (AP) in the low Mach limit and strictly positivity preserving for density and potential temperature. The scheme ensures stability and accuracy across a broad range of Mach numbers, from fully compressible to nearly incompressible regimes. We rigorously establish consistency with both the compressible system and its incompressible, density-dependent limit. Numerical experiments confirm that the method robustly captures complex flow features while preserving the essential physical and mathematical structures of the model.
\end{abstract}

\maketitle

\section{Introduction}
\label{sec:intro}
The compressible Euler equations with potential temperature transport are widely employed in meteorological modelling to describe the dynamics of inviscid, compressible fluids under adiabatic conditions, neglecting molecular transport effects on potential temperature \cite{Kle10}. In this work, we propose an all-speed, asymptotic preserving (AP), semi-implicit finite volume scheme for solving these equations. The scheme is designed to operate robustly across a wide spectrum of flow regimes—from highly compressible flows with sharp discontinuities to weakly compressible, low Mach number flows, where fluid velocities are much smaller than the speed of sound. Such regimes commonly arise in applications ranging from meteorology to combustion and astrophysics. Accurately and efficiently simulating these flows remains a significant challenge due to the stiffness of the governing equations in the low Mach number limit. 

In the literature, a wide range of numerical methods has been developed to accurately simulate low Mach number flows. One prominent approach involves extending the MAC-type finite difference scheme, originally formulated for incompressible flows, into a conservative pressure-correction framework on staggered grids \cite{BW98}. The widely used SIMPLE algorithm has also been adapted for weakly compressible regimes, employing formal asymptotic analysis to clarify the limiting behaviour as the Mach number tends toward zero \cite{MRK+03}. This has given rise to numerical frameworks featuring multiple pressure variables, enabling the accurate representation of multiscale physical phenomena \cite{PM05}. Higher-order discontinuous Galerkin methods have also been applied within the context of all-Mach-number flow simulations \cite{FDK07}. An alternative strategy involves the use of the so-called ``sound-proof" models, which eliminate acoustic wave propagation and are particularly effective in simulating low Mach number atmospheric dynamics \cite{SKW14}. Furthermore, various efforts have focused on adapting compressible flow solvers to the low Mach regime, typically leveraging asymptotic expansions in terms of the Mach number to improve accuracy and stability; see, e.g., \cite{BAL+14, DT11, HJL12, Kle95} and the references therein.

An effective and widely adopted framework for designing robust numerical approximations to the low Mach number Euler equations—or more generally, to singularly perturbed hydrodynamic models—is the Asymptotic Preserving (AP) methodology. Originally introduced in the context of kinetic models for diffusive transport \cite{Jin99}, AP schemes offer a systematic approach for numerically addressing singular perturbation problems. The core idea can be described as follows. Let $\mcal{P}^\veps$ represent a family of problems dependent on a small perturbation parameter $\veps$, with the property that as $\veps \to 0$, the solution of $\mcal{P}^\veps$ converges to that of a well-posed limit problem $\mcal{P}^0$. A numerical approximation $\mathcal{P}^\varepsilon_h$, where $h$ denotes the discretisation parameter, is said to be AP if it satisfies the following two conditions: (1) as $\varepsilon \to 0$, the numerical scheme $\mathcal{P}^\varepsilon_h$ converges to a scheme $\mathcal{P}^0_h$ that is a consistent discretisation of the limiting system $\mathcal{P}^0$; and (2) the stability constraints on $h$ remain uniform, i.e., independent of $\varepsilon$. The AP methodology is particularly well-suited for capturing the zero Mach number limit, as it ensures consistency with the asymptotic behaviour of the governing equations at the discrete level. Moreover, AP schemes inherently distinguish between singular (weakly compressible or nearly incompressible), non-singular (fully compressible), and transitional flow regimes. This allows them to adapt seamlessly to different regions within the computational domain. As a result, AP discretisations not only preserve the correct asymptotic behaviour but also significantly reduce computational cost and improve overall accuracy in the simulation of low Mach number flows.

One of the most effective strategies for constructing Asymptotic Preserving (AP) schemes for hydrodynamic models, such as the compressible Euler equations, is the use of semi-implicit or implicit-explicit (IMEX) time integration. This technique has been widely adopted in the numerical treatment of stiff hyperbolic systems, with significant developments reported in \cite{AS20, BAL+14, BJR+19, CGK16, DT11, HJL12, HLS21, NBA+14, TPK20}. The core idea of IMEX schemes is to separate the governing equations into stiff and non-stiff components. The stiff terms, typically responsible for rapid acoustic wave propagation, are treated implicitly to relax restrictive time-step limitations, while the remaining non-stiff terms are handled explicitly to maintain computational efficiency. This semi-implicit approach offers a practical compromise between the high cost of fully implicit schemes and the stringent stability constraints of explicit methods. In the context of barotropic Euler equations, an AP IMEX scheme usually requires implicit treatment of both the mass flux and the pressure gradient within the momentum equations \cite{DT11}. This formulation leads naturally to an elliptic equation for the pressure, which can be interpreted as a time-discrete analogue of the acoustic wave equation in the continuous regime. Despite its advantages, this elliptic formulation introduces numerical challenges. The discrete diffusion operator formed from the combination of discrete gradient and divergence operators often fails to satisfy the inf-sup (or Ladyzhenskaya–Babu{\v s}ka–Brezzi) condition. As a result, additional stabilisation techniques are often required to ensure numerical robustness in the low Mach number limit \cite{MKB+12}. A promising solution to this issue is the extension of well-established numerical techniques from incompressible flow solvers to the compressible setting. One of the earliest examples is the extension of the Marker-and-Cell (MAC) scheme to compressible flows \cite{HA68}, building on the classical method developed in \cite{HW65}. Further refinements of this approach can be found in \cite{BW98, CP99, IGW86, KP89, KSG+09, WPM02}, where staggered-grid discretisations and pressure projection methods are adapted to preserve stability and accuracy across a wide range of flow regimes. An additional and equally important requirement in the numerical treatment of compressible fluid models is entropy stability. As solutions may develop discontinuities such as shock waves in finite time, enforcing discrete entropy inequalities is essential for ensuring the selection of physically admissible weak solutions. Although constructing entropy-stable schemes is a nontrivial task, it is indispensable in low Mach number simulations, where such stability is key to deriving rigorous energy estimates and performing asymptotic convergence analyses.

This work aims to design and analyse a minimally semi-implicit, energy-stable scheme for the Euler system with temperature transport on a MAC grid. Energy stability is achieved by introducing a velocity shift in the convective fluxes of the potential temperature, inspired by \cite{CDV17, DVB20, PV16}. This shift, proportional to the stiff pressure gradient, dissipates mechanical energy uniformly across Mach numbers. A semi-implicit time discretisation is only employed for the transport of potential temperature, which alleviates stiffness constraints and enforces the divergence condition on the velocity field. The mass update is explicit, while the implicit temperature update leads to a semilinear elliptic problem for pressure, which is well-posed thanks to the MAC grid structure. After solving this, the momentum update is computed explicitly.
The scheme ensures the existence of numerical solutions, as well as the positivity of mass and temperature, and exhibits the correct AP behaviour in the zero Mach number limit. It is also weakly consistent with the governing equations. A derived sufficient stability condition confirms that the scheme supports large time steps in the low Mach number regime. The scheme’s structure-preserving and AP properties—namely, uniform time-step bounds and consistency with the incompressible limit—are also numerically validated through a series of test cases.

The remainder of this paper is organised as follows. Section~\ref{sec:ept_cont-case} introduces the scaled compressible Euler system with potential temperature transport, its zero Mach number limit, and the stabilisation technique. In Section~\ref{sec:ept_upwind_scheme}, we present the numerical scheme along with its stability properties. Sections~\ref{sec:ept_cons_LW} and~\ref{sec:ept_cons_incomp} are devoted to establishing the two main consistency results. Numerical experiments are reported in Section~\ref{sec:ept_num_res}, and concluding remarks are given in Section~\ref{sec:ept_conclusion}. 

\section{Euler System with Potential Temperature Transport and Incompressible Limit}
\label{sec:ept_cont-case}
We start with the following initial boundary value problem for the compressible Euler system with potential temperature transport, parametrised by $\veps$:
\begin{subequations}
\label{eq:euler_pot-temp}
    \begin{gather}
       \D_t\rho^\veps+\dive (\rho^\veps\uu{u}^\veps)=0, \label{eq:ept_cons_mas}\\ 
       \D_t(\rho^\veps\uu{u}^\veps)+\bdive(\rho^\veps\uu{u}^\veps\otimes\uu{u}^\veps)+\frac{1}{\veps^2}\bgrd p^\veps =0, \label{eq:ept_cons_mom} \\
       \D_t(\rho^\veps\theta^\veps)+\dive (\rho^\veps\theta^\veps\uu{u}^\veps)=0, \label{eq:ept_cons_temp}\\
       \rho^\veps\vert_{t=0}=\rho^\veps_0, \quad
       \uu{u}^\veps\vert_{t=0}=\uu{u}^\veps_0, \quad \theta^\veps\vert_{t=0}=\theta^\veps_0, \label{eq:ept_eq_ic} 
   \end{gather}
\end{subequations}
for $(t,\uu{x})\in Q:=(0,T)\times\Omega$, where $T>0$ is a given time and $\Omega$ is an open, bounded and connected subset of $\mbb{R}^d$, $d\geqs 1$. The
infinitesimal $\veps\in(0,1]$ is called the reference Mach number, and is defined as the ratio of a reference fluid speed to that of a reference sound speed. The variables $\rho^\veps>0$, $\uu{u}^\veps=(u_1^\veps,\dots,u_d^\veps)$ and $\theta^\veps>0$ are the density, the velocity and the potential temperature of the fluid, respectively. The pressure $p^\veps=p(\rho^\veps\theta^\veps)$ is assumed to follow an equation of state $p(\rho\theta):=(\rho\theta)^\gm$ with $\gm > 1$ being the adiabatic index. The variable $\rho^\veps\theta^\veps$ represents the total potential temperature of the fluid. For ease of analysis, we impose spatially periodic boundary conditions, thereby identifying $\Omega$ with the $d$-dimensional torus $\mbb{T}^d$ throughout the paper.

Smooth solutions of the system \eqref{eq:euler_pot-temp} are known to satisfy certain a priori energy estimates. Combined with appropriate assumptions on the initial data, these estimates provide a foundation for justifying the incompressible limit as $\veps \to 0$.

\subsection{A Priori Energy Estimates}
\label{subsec:ept_engy_est}
We start by defining the notion of the internal energy per unit volume or the so-called Helmholtz function:
\begin{equation}
   \label{eq:ept_psi_gamma}
   \psi_\gm(\rho\theta) :=
     \dfrac{(\rho\theta)^\gm}{\gm-1}, \ \gm>1.  
\end{equation}
\begin{proposition}
  \label{prop:ept_eng_balance}
  The smooth solutions of \eqref{eq:euler_pot-temp} satisfy the following key identities:
  \begin{enumerate}[label=(\roman*)]
  \item a renormalisation identity:
    \begin{equation}
      \label{eq:ept_renorm}
      \Dt\psi_\gm(\rho^\veps\theta^\veps)
      +\dive\big(\psi_\gm(\rho^\veps\theta^\veps)\uu{u}^\veps\big)+p^\veps\dive\uu{u}^\veps=0,
    \end{equation}
  \item a positive renormalisation identity:
    \begin{equation}
      \label{eq:ept_porenorm}
      \Dt\Pi_\gm(\rho^\veps\theta^\veps)
      +\dive\big(\psi_\gm(\rho^\veps\theta^\veps)-\psi_\gm^\prime(1)\rho^\veps\theta^\veps\big)\uu{u}^\veps+p^\veps\dive\uu{u}^\veps=0,
    \end{equation}
    where $\Pi_\gm(z):=\psi_\gm(z)-\psi_\gm(1)-\psi_{\gm}^{\prime}(1)(z-1)$ is the relative internal energy, which is an affine approximation of $\psi_\gm$ with respect to the constant state $z=1$,
  \item a kinetic energy identity:
    \begin{equation}
      \label{eq:ept_kinbal}
      \Dt\Big(\frac{1}{2}\rho^\veps{\abs{\uu{u}^\veps}}^2\Big)
      +\dive\Big(\frac{1}{2}\rho^\veps{\abs{\uu{u}^\veps}}^2\uu{u}^\veps\Big)
      +\frac{1}{\veps^2}\grd p^\veps\cdot\uu{u}^\veps
      =0,
    \end{equation}
  \item a total energy identity:
  \begin{equation}
  \label{eq:ept_tot_eng_id}
  \Dt\Big(\frac{1}{2}\rho^\veps{\abs{\uu{u}^\veps}}^2+\frac{1}{\veps^2}\Pi_\gm(\rho^\veps\theta^\veps)\Big)
      +\dive\Big(\frac{1}{2}\rho^\veps{\abs{\uu{u}^\veps}}^2+\frac{1}{\veps^2}\psi_\gm(\rho^\veps\theta^\veps)-\frac{1}{\veps^2}\psi_\gm^{\prime}(1)\rho^\veps\theta^\veps+\frac{1}{\veps^2}p^\veps\Big)\uu{u}^\veps=0.      
  \end{equation}
  \end{enumerate}
\end{proposition}

\subsection{Incompressible Limit}
\label{sec:ept_incomp_lim}
This section aims to analyse the incompressible limit of the system \eqref{eq:euler_pot-temp}, obtained by letting the Mach number $\varepsilon$ tend to zero along a sequence of weak solutions $(\rho^\varepsilon, \uu{u}^\varepsilon, \theta^\varepsilon)_{\varepsilon > 0}$. The a priori estimates established for smooth solutions in the preceding section—particularly the total energy estimate \eqref{eq:ept_tot_eng_id}—provide uniform bounds independent of $\varepsilon$, forming the foundation for extending the analysis to weak solutions. While the total energy estimate \eqref{eq:ept_tot_eng_id} holds as an equality for smooth solutions, it is assumed to hold as an inequality in the weak formulation. To proceed, we assume the existence of a weak solution $(\rho^\varepsilon, \uu{u}^\varepsilon, \theta^\varepsilon) \in L^\infty(Q)^{d+2}$ to the system \eqref{eq:euler_pot-temp}, satisfying the following inequality:
\begin{equation}
 \label{eq:ept_tot_eng_ineq_cont}
 \Dt\Big(\frac{1}{2}\rho^\veps{\abs{\uu{u}^\veps}}^2+\frac{1}{\veps^2}\Pi_\gm(\rho^\veps\theta^\veps)\Big)
       +\dive\Big(\frac{1}{2}\rho^\veps{\abs{\uu{u}^\veps}}^2+\frac{1}{\veps^2}\psi_\gm(\rho^\veps\theta^\veps)-\frac{1}{\veps^2}\psi_\gm^{\prime}(1)\rho^\veps\theta^\veps+\frac{1}{\veps^2}p^\veps\Big)\uu{u}^\veps\leqs 0.      
 \end{equation}
 An integration of \eqref{eq:ept_tot_eng_ineq_cont} over $\Omega$, along with periodic boundary conditions on the velocity, yields the following total energy estimate:
 \begin{equation}
 \label{eq:ept_tot_eng_est_cont}
     \half\int_\Omega\rho^{\veps}(t)\absq{\uu{u}^\veps(t)}\dd\uu{x}+\frac{1}{\veps^2}\int_\Omega\Pi_\gm(\rho^\veps(t)\theta^\veps(t))\dd\uu{x}\leqs\half\int_\Omega\rho^{\veps}_0\absq{\uu{u}^\veps_0}\dd\uu{x}+\frac{1}{\veps^2}\int_\Omega\Pi_\gm(\rho^\veps_0\theta^\veps_0)\dd\uu{x}.
 \end{equation}
 The left-hand side of the energy estimate \eqref{eq:ept_tot_eng_est_cont} remains uniformly bounded, provided that the right-hand side—which depends on the initial data—is uniformly bounded with respect to $\varepsilon$. We show that this condition holds for both well-prepared and ill-prepared initial data. In the following, we present the definition of these data types, inspired by \cite{HLS21}.
  \begin{definition}
    \label{def:ept_ill_prep_id}
    An initial data $(\rho^\veps_0,\uu{u}^{\veps}_0, \theta^\veps_0)$ of \eqref{eq:euler_pot-temp} is called ill-prepared, if $(\rho^{\veps}_{0},\uu{u}^{\veps}_{0}, \theta^\veps_0)\in L^{\infty}(\Omega)^{d+2}$ with $\rho^{\veps}_{0}>0$, $\theta^\veps_{0}>{c}$ and satisfy the bound:
    \begin{equation}
    \label{eq:ept_ill_prep_id}
        \norm{\uu{u}_{0}^{\veps}}_{L^{2}(\Omega)^d}+\frac{1}{\veps}\norm{\rho_{0}^{\veps}\theta_{0}^{\veps}-1}_{L^{\infty}(\Omega)}\leqs C,
    \end{equation}
    where ${c}>0$ and $C>0$ are constants independent of $\veps$. 
\end{definition}
\begin{remark}
    The estimate \eqref{eq:ept_ill_prep_id} indicates that $\rho_0^\veps\theta_{0}^{\veps}-1=\mcal{O}(\veps)$ and that $\uu{u}^\veps_0$ is uniformly bounded in $L^{2}(\Omega)^d$ with respect to $\veps$. In contrast, the so-called `well-prepared' data, as considered in the literature, e.g., \cite{KM82,Sch94}, requires more stringent restrictions such as $\rho_{0}^{\veps}\theta_{0}^{\veps}-1=\mcal{O}(\veps^2)$, a uniform bound on $\uu{u}_0^\veps$ in $H^1(\Omega)^d$ and $\dive\uu{u}_0^\veps$ being close to zero, i.e., the velocity field $\uu{u}_0^\veps$ is nearly divergence-free.
\end{remark}

Along with some estimates on $\Pi_\gm$ (see \cite[Lemma 2.3]{HLS21}), it follows that the right side of the energy estimate \eqref{eq:ept_tot_eng_est_cont} remains uniformly bounded even when the initial data are ill-prepared. Consequently, the boundedness of the left side of \eqref{eq:ept_tot_eng_est_cont} implies the following strong convergence:
\begin{equation}
\label{eq:ept_rho-theta_conv}
    \rho^\veps \theta^\veps \to 1 \quad \text{in } L^\infty(0, T; L^r(\Omega)) \quad \text{for every } r \in [1, \min\{2, \gm\}].
\end{equation}
Further, under the additional assumption that the density is uniformly bounded away from zero, the weak-* convergence
\begin{equation}
\uu{u}^\veps\overset{\ast}{\rightharpoonup}\uu{U} \quad \text{in } L^\infty(0, T; L^2(\Omega)^d)
\tag{2.11}
\end{equation}
holds as $\veps\to0$. The strong convergence of $\rho^\veps$ and $\theta^\veps$, which are needed to pass to the limit $\veps\to 0$ in the system \eqref{eq:euler_pot-temp}, is one of the difficult parts of the asymptotic analysis. Assuming a uniform bound on $\theta^\veps$, together with the arguments developed in \cite{Lio98}, it can be shown that $\theta^\veps\to\theta$ in $L^\infty(0,T;L^q(\Omega))$ for all $1\leqs q<\infty$. Consequently, using \eqref{eq:ept_rho-theta_conv}, it follows $\rho^\veps\to\rho$ in $L^\infty(0,T;L^r(\Omega))$ with $\rho\theta=1$. Thus, we expect that the limiting solution $(\rho,\uu{U}, \pi)\in L^\infty(0,T;L^2(\Omega))^{d+2}$ is a weak solution of the initial value problem
 \begin{subequations}
 \label{eq:incomp_den-dep_sys}
     \begin{gather}
        \Dt\rho+\dive(\rho\uu{U})=0, \\
        \Dt(\rho\uu{U})+\uu{\dive}(\rho\uu{U}\otimes\uu{U})+\grd \pi =0, \\
        \dive \uu{U}=0,\\
        \rho(0, \cdot)=\rho_0=\frac{1}{\theta_0}, \ \uu{U}(0,\cdot)=\uu{U}_0,
\end{gather}
\end{subequations}
where $\theta_0$ is the strong limit of $\theta^\veps_0$ in $L^q(\Omega)$ for all $1\leqs q<\infty$, and $\pi$ is the formal limit of $\frac{p(\rho^\veps\theta^\veps)-1}{\veps^2}$. The system \eqref{eq:incomp_den-dep_sys}, commonly referred to as the incompressible density-dependent Euler equations \cite{Lio96}, is the formal limit of \eqref{eq:euler_pot-temp} as $\veps \to 0$. The literature dealing with the incompressible limit of \eqref{eq:euler_pot-temp} is rather sparse; see, e.g., \cite{LM98} for the analysis of the incompressible limit of the Navier-Stokes system with potential temperature transport. 

\subsection{Velocity Stabilisation}
\label{sec:stab}
To ensure numerical energy stability, we adopt the stabilisation strategy proposed in \cite{CDV17, DVB20, GVV13}, which involves introducing a shifted velocity into the convective fluxes of the mass and momentum equations. In contrast, our formulation of system \eqref{eq:euler_pot-temp} necessitates the use of a shifted velocity solely in the total potential temperature flux within \eqref{eq:ept_cons_temp}. This leads to the following modified system:
\begin{subequations}
\label{eq:r_euler_pot-temp}
\begin{gather}
  \D_t\rho^\veps+\dive (\rho^\veps\uu{u}^\veps)=0, \label{eq:ept_r_cons_mas}
  \\ 
  \D_t(\rho^\veps\uu{u}^\veps)+\bdive
  (\rho^\veps\uu{u}^\veps\otimes\uu{u}^\veps)+\frac{1}{\veps^2}\bgrd
  p^\veps =0. \label{eq:ept_r_cons_mom}
  \\
  \D_t(\rho^\veps\theta^\veps)+\dive (\rho^\veps\theta^\veps(\uu{u}^\veps-\delta\uu{u}^\veps))=0. \label{eq:ept_r_cons_temp}
\end{gather}
\end{subequations}
This system is supplemented with the same initial and boundary conditions as those of the original system \eqref{eq:euler_pot-temp}. Analogous to Proposition~\ref{prop:ept_eng_balance}, the solutions of the modified system \eqref{eq:r_euler_pot-temp} can be shown to satisfy the following a priori estimates.
\begin{proposition}
  \label{prop:ept_r_engy_balance}
  Any smooth solution of \eqref{eq:r_euler_pot-temp} satisfy 
  \begin{enumerate}[label=(\roman*)]
  \item a renormalisation identity:
    \begin{equation}
      \label{eq:rpt_r_renorm}
      \Dt\psi_\gm(\rho^\veps\theta^\veps)
      +\dive\left(\psi_\gm(\rho^\veps\theta^\veps)(\uu{u}^\veps-\delta\uu{u}^\veps)\right)+p^\veps\dive(\uu{u}^\veps-\delta\uu{u}^\veps)=0,
    \end{equation}
  \item a positive renormalisation identity:
    \begin{equation}
      \label{eq:ept_r_porenorm}
     \Dt\Pi_\gm(\rho^\veps\theta^\veps)
      +\dive\big(\psi_\gm(\rho^\veps\theta^\veps)-\psi_\gm^\prime(1)\rho^\veps\theta^\veps\big)(\uu{u}^\veps-\delta\uu{u}^\veps)+p^\veps\dive(\uu{u^\veps}-\delta\uu{u}^\veps)=0,
    \end{equation}
  \item a kinetic energy identity:
    \begin{align}
      \label{eq:ept_r_kinbal}
      \Dt\Big(\frac{1}{2}\rho^\veps{\abs{\uu{u}^\veps}}^2\Big)
      +\dive\Big(\frac{1}{2}\rho^\veps{\abs{\uu{u}^\veps}}^2\uu{u}^\veps\Big)
      +\frac{1}{\veps^2}(\uu{u}^\veps-\delta\uu{u}^\veps)\cdot\bgrd p^\veps=-\frac{1}{\veps^2}\delta\uu{u}^\veps\cdot\bgrd p^\veps,
    \end{align}
  \item a total energy balance by adding $\frac{1}{\veps^2}$ times \eqref{eq:ept_r_porenorm} to \eqref{eq:ept_r_kinbal}:
  \begin{gather}
      \Dt\Big(\frac{1}{\veps^2}\Pi_\gm(\rho^\veps\theta^\veps)+\frac{1}{2}\rho^\veps{\abs{\uu{u}^\veps}}^2\Big)+\dive\Big(\frac{1}{2}\rho^\veps{\abs{\uu{u}^\veps}}^2\uu{u}^\veps+\big(\frac{1}{\veps^2}\psi_\gm(\rho^\veps\theta^\veps)-\frac{1}{\veps^2}\psi_\gm^\prime(1)\rho^\veps\theta^\veps+\frac{1}{\veps^2}p^\veps\big)(\uu{u}^\veps-\delta\uu{u}^\veps)\Big)\notag\\
      =-\frac{1}{\veps^2}\delta\uu{u}^\veps\cdot\bgrd p^\veps\label{eq:ept_r_eng_id}.
  \end{gather}
  \end{enumerate}
\end{proposition}
\begin{proof}
    The uses straightforward manipulations as in Proposition~\ref{prop:ept_eng_balance}.
\end{proof}
The total energy balance \eqref{eq:ept_r_eng_id} naturally suggests to take $\delta\uu{u}^\veps=\eta\bgrd p^\veps$ with $\eta>0$, in order to obtain the following energy stability inequality:
\begin{gather}
    \label{eq:ept_r_eng_id_stab}
    \Dt\Big(\frac{1}{\veps^2}\Pi_\gm(\rho^\veps\theta^\veps)+\frac{1}{2}\rho^\veps{\abs{\uu{u}^\veps}}^2\Big)+\dive\Big(\frac{1}{2}\rho^\veps{\abs{\uu{u}^\veps}}^2\uu{u}^\veps+\big(\frac{1}{\veps^2}\psi_\gm(\rho^\veps\theta^\veps)-\frac{1}{\veps^2}\psi_\gm^\prime(1)\rho^\veps\theta^\veps+\frac{1}{\veps^2}p^\veps\big)(\uu{u}^\veps-\delta\uu{u}^\veps)\Big)\notag\\
      =-\frac{1}{\veps^2}\delta\uu{u}^\veps\cdot\bgrd p^\veps=-\frac{1}{\veps^2}\eta\abs{\bgrd p^\veps}^2 \leqs 0. 
\end{gather}
In other words, the introduction of a velocity shift exerts a stabilising influence, ultimately leading to the desired energy stability inequality \eqref{eq:ept_tot_eng_est_cont}. As a consequence of this energy stability, and in light of the formal asymptotic analysis presented in the preceding section, the modified system \eqref{eq:r_euler_pot-temp} formally converges, in the limit $\veps \to 0$, to the following velocity-stabilised, incompressible, density-dependent Euler equations:
\begin{subequations}
\label{eq:r_incomp_euler_dd}
\begin{gather}   
    \D_t\rho+\dive(\rho\uu{U})=0,\label{eq:t_incomp_eul_den}\\    
    \D_t(\rho\uu{U})+\bdive(\rho\uu{U}\otimes\uu{U})+\bgrd\pi=0,\label{eq:t_incomp_eul_pres}\\
    \dive(\uu{U}-\delta\uu{U})=0,\label{eq:r_incomp_eul_vel}
\end{gather}
\end{subequations}
where the limiting stabilisation term is given by $\delta\uu{U}=\eta\bgrd\pi$.

Motivated by the above analysis, we construct an energy-stable numerical scheme by incorporating a velocity shift into the total potential temperature flux. This approach enables us to derive a discrete analogue of the incompressible limit.

\section{An Energy Stable Semi-implicit Upwind Scheme}
\label{sec:ept_upwind_scheme}

In the following, we introduce our semi-implicit in time, upwind in space, fully-discrete scheme for the Euler system \eqref{eq:r_euler_pot-temp}. To begin with, we discretise the domain $\Omega$ using a MAC grid, consisting of a pair $(\M,\E)$, where $\M$ represents the primal mesh which is a partition of $\bar{\Omega}$ composing of possibly non-uniform closed rectangles ($d=2$) or parallelepipeds ($d=3$) and $\E$ denotes the set of all edges of the primal cells. A generic element of $\M$ is denoted by $K$, known as a primal cell, and a generic element of $\E$ is represented by $\s$.
Also, we have $\mcal{E}=\mcal{E}_\intr\cup\mcal{E}_\extr$, where $\mcal{E}_\intr$ and $\mcal{E}_\extr$ are, respectively, the collection of internal and external edges of $\E$. We denote by $\E^{(i)}$ the set of $d-1$ dimensional edges that are orthogonal to $\uu{e}^{(i)}$ and decompose $\E^{(i)}$ as $\mcal{E}^{(i)}=\mcal{E}_\intr^{(i)}\cup\mcal{E}_\extr^{(i)}$, where $\E_\intr^{(i)}$ (resp. $\E_\intr^{(i)}$) are, respectively, the internal and external edges of $\mcal{E}^{(i)}$. We denote $\s=K|L$, where the edge $\s\in\E_\intr$ is such that $\s=\bar{K}\cap\bar{L}$ with $K,L\in\M$. For each edge $\s\in\E$, the dual cell $\Ds$ is defined as follows: if $\s$ is an internal edge shared by two cells, $\Ds$ is the union of the adjacent halves from both cells; if $\s$ lies in a boundary, $\Ds$ is the half of the adjacent cell near $\s$. Furthermore, we denote by $\mcal{E}(K)$, the set of all edges of $K\in\mcal{M}$ and by $\tilde{\E}(\Ds)$, the set of all edges of the dual cell $\Ds$. We denote by $L_{\M}(\Omega)$, the space of scalar-valued functions which are piecewise constant on each primal cell $K\in\M$. The space $L_\M$ is used for approximating the density and the potential temperature. Analogously, we define by $\uu{H}_{\E}(\Omega)=\prod_{i=1}^{d} H^{(i)}_{\E}(\Omega)$, the set of vector-valued (in $\mbb{R}^d$) functions which are constant on each dual cell $\Ds$ and for each $i=1,2,\dots,d$. The space of vector-valued functions vanishing on the external edges is denoted as  $\uu{H}_{\E,0}(\Omega)=\prod_{i=1}^d H^{(i)}_{\E,0}(\Omega)$, where  $H^{(i)}_{\E,0}(\Omega)$ contains those elements of $H^{(i)}_{\E}(\Omega)$ which vanish on the external edges.

We closely adhere to the notation for finite volume discretisation introduced in \cite{AGK23}, to which we refer the reader for further details regarding the mesh structure and the discrete unknowns outlined above. The discrete differential operators used to approximate those in the system \eqref{eq:euler_pot-temp} are likewise adopted from \cite{AGK23}; therefore, we omit their detailed description here for brevity.

\subsection{The Scheme}
\label{sec:ept_scheme}
Let us consider a discretisation $0=t^0<t^1<\cdots<t^N=T$ of the time-interval $(0,T)$ and let $\dt=t^{n+1}-t^n$, for $n=0,1,\dots,N-1$, be the constant time-step. 

In what follows, we suppress the $\veps$-dependence of the discrete unknowns, except where explicitly stated. We consider the initial approximations for $\rho$ and $\theta$ as the averages of the initial conditions $\rho^\veps_{0}$ and $\theta^\veps_0$ on the primal cells. Analogously, we take the initial approximation for $\uu{u}$ as the average of the initial data $\uu{u}^\veps_{0}$ on the dual cells, i.e.\ 
\begin{subequations}
\label{eq:ept_dis_ic}
\begin{gather}
    \rho_{K}^{0}=\frac{1}{|K|}\int_{K}\rho^\veps_{0}(\uu{x})\dd\uu{x}, \; \forall K\in\M,\\
    \theta_{K}^{0}=\frac{1}{|K|}\int_{K}\theta^\veps_{0}(\uu{x})\dd\uu{x}, \; \forall K\in\M,\\
    u_{\s}^{0}=\frac{1}{|\Ds|}\int_{\Ds}(\uu{u}^\veps_{0}(\uu{x}))_{i}\dd\uu{x}, \; \forall \s\in\E_{\intr}^{(i)}, \, 1\leqs i\leqs d. 
\end{gather}
\end{subequations}

We consider the following fully-discrete scheme for $0\leqs n\leqs{N-1}$:
\begin{subequations}
\label{eq:ept_dis_update}
    \begin{gather}
        \frac{1}{\dt}\big(\rho_{K}^{n+1}-\rho_{K}^{n}\big)+\frac{1}{\left|K\right|}\sum_{\s\in\E(K)}F_{\s,K}(\rho^{n},\uu{u}^{n})=0, \ \forall K\in \M,\label{eq:ept_dis_cons_mas}\\ 
        \frac{1}{\dt}\big(\rho_{\Ds}^{n+1}u_\s^{n+1}-\rho_{\Ds}^{n}u_\s^{n}\big)+\frac{1}{\left|\Ds\right|}\sum_{\epsilon\in\tilde{\E}(\Ds)}F_{\epsilon,\s}(\rho^{n},\uu{u}^{n})u_{\eps,\mathrm{up}}^{n}+\frac{1}{\veps^2}(\partial^{(i)}_{\E}p^{n+1})_{\s}=0, \ 1\leqs i\leqs d, \ \forall \s\in\E_\intr^{(i)},\label{eq:ept_dis_cons_mom}\\
        \frac{1}{\dt}\big(\rho_{K}^{n+1}\theta_{K}^{n+1}-\rho_{K}^{n}\theta_{K}^{n}\big)+\frac{1}{\left|K\right|}\sum_{\s\in\E(K)}F_{\s,K}(\rho^{n+1}\theta^{n+1},\uu{v}^{n+1})=0, \ \forall K\in \M.\label{eq:ept_dis_cons_temp}
\end{gather}
\end{subequations}
where the mass flux is defined by
\begin{equation}
      \label{eq:ept_mass_flux}
      F_{\s,K}(\rho^n,\uu{u}^n)=\abs{\s}\big\{\rho_{K}^{n}u_{\s,K}^{n,+}+\rho_{L}^{n}u_{\s,K}^{n,-}\big\}.
\end{equation}
Here $u_{\s , K}=u_{\s} \uu{e}^{(i)}\cdot\uu{\nu}_{\s, K}$ is an approximation of the normal velocity on $\s$ and $\uu{\nu}_{\s, K}$ is the unit vector normal to the edge $\s\in\mcal{E}^{(i)}_{\intr}\cap\Ek$ in the direction outward to the cell $K$. The temperature flux is defined by
\begin{equation}
      \label{eq:ept_temp_flux}
      F_{\s,K}(\rho^{n+1}\theta^{n+1},\uu{v}^{n+1})=\abs{\s}\big\{\rho_{K}^{n+1}\theta_{K}^{n+1}(v_{\s,K}^{n+1})_{+}+\rho_{L}^{n+1}\theta_{L}^{n+1}(v_{\s,K}^{n+1})_{-}\big\}.
\end{equation}
In the above, the stabilised velocity $\uu{v}^{n+1}$ is written as $v_{\s}^{n+1}=u_{\s}^{n}-\delta u_{\s}^{n+1}$ with the stabilisation term defined by
\begin{equation}
    \delta u_{\s}^{n+1}=\frac{\eta\dt}{\veps^2}(\partial^{(i)}_{\E}p^{n+1})_{\s},
\end{equation}
where $\eta>0$ is to be chosen after a stability analysis of the scheme.
To retain an upwind bias and achieve sign splitting, the positive and negative valued parts of $v_{\s,K}^{n+1}$ used in \eqref{eq:ept_temp_flux} are defined by
\begin{align*}
(v^{n+1}_{\s,K})_+ &= \max\{u^{n}_{\s,K}, 0\} - \min\{\delta u^{n+1}_{\s,K}, 0\} = u_{\s,K}^{n,+} - \delta u^{n+1, -}_{\s,K} \geqs 0, \\
(v^{n+1}_{\s,K})_- &= \min\{u^{n}_{\s,K}, 0\} - \max\{\delta u^{n+1}_{\s,K}, 0\} = u_{\s,K}^{n,-} - \delta u^{n+1,+}_{\s,K} \leqs 0.
\end{align*}

\subsection{Existence of a Numerical Solution and Positivity of the Density}
\label{sec:exist_soln}
The mass update \eqref{eq:ept_dis_cons_mas} is explicit, so $\rho^{n+1}$ can be evaluated explicitly. But, the potential temperature update \eqref{eq:ept_dis_cons_temp} is nonlinear in $\rho^{n+1}\theta^{n+1}$ due to the stabilisation term. However, once $\rho^{n+1}\theta^{n+1}$ is calculated, the momentum update \eqref{eq:ept_dis_cons_mom} can be explicitly evaluated to get the velocity. In what follows, we establish the existence of a discrete solution to the numerical scheme \eqref{eq:ept_dis_update}. The result also yields the positivity of the density and the temperature. Our treatment is analogous to the one in \cite{GMN19, NIR01, DYY06}, which uses classical tools from topological degree theory in finite dimensions \cite{Dei85}.
\begin{theorem}
    \label{thm:ept_existence}
    Let $(\rho^n,\uu{u}^n, \theta^n)\in L_{\mcal{M}}(\Omega)\times\uu{H}_{\mcal{E},0}(\Omega)\times L_{\mcal{M}}(\Omega)$ be such that $\rho^{n}>0$, $\theta^{n}>0$ on $\Omega$. Then, under a timestep condition
    \begin{equation}
        \label{eq:ept_time-step_positivity}
        \frac{\dt}{\abs{K}}\sum\limits_{\s\in\Ek}\frac{|F_{\s,K}(\rho^n,\uu{u}^n)|}{\rho_K^n}\leqslant \frac{1}{3},
    \end{equation}
    there exists a solution $(\rho^{n+1}, \uu{u}^{n+1}, \theta^{n+1})\in L_{\mcal{M}}(\Omega)\times\uu{H}_{\mcal{E},0}(\Omega)\times L_{\mcal{M}}(\Omega)$ of  \eqref{eq:ept_dis_update}, satisfying $\rho^{n+1}>0$ and $\theta^{n+1}>0$ on $\Omega$.
\end{theorem}
\begin{proof}
From the explicit mass update \eqref{eq:ept_dis_cons_mas} using upwind fluxes, the positivity of the initial density and the timestep condition \eqref{eq:ept_time-step_positivity}, the existence and positivity of $\rho^{n+1}$ follow directly. Similarly, the momentum update gives the existence of $\uu{u}^{n+1}$ once the pressure $p^{n+1}$ is known. Therefore, it only remains to show the existence and positivity of $\theta^{n+1}$. To this end, we introduce an open, bounded subset $V$ of $L_{\mcal{M}}(\Omega)$ via 
\begin{equation*}
    V=\bigg\{\Theta=\sum_{K\in\mcal{M}}\Theta_{K}\mcal{X}_{K}\in L_{\mcal{M}}(\Omega)\colon 0<\Theta_{K}<C, \ \forall K\in\mcal{M}\bigg\},
\end{equation*}
where $C>0$ is a constant such that 
\begin{equation}
\label{eq:C_thm_ex}
    C>\frac{\abs{\Omega}\max_{K\in\mcal{M}}\{\rho_{K}^n\theta_{K}^n\}}{\min_{K\in\mcal{M}}\{\abs{K}\}}.
\end{equation}
Let us consider a function $H\colon[0,1]\times L_{\mcal{M}}(\Omega)\rightarrow L_{\mcal{M}}(\Omega)$, defined by 
$H(\lambda,\Theta)=\sum_{K\in\mcal{M}}H(\lambda,\Theta)_{K}\mcal{X}_{K}$, where
\begin{equation*}
    H(\lambda,\Theta)_{K}=\frac{1}{\dt}(\Theta_{K}-\rho^n_{K}\theta^n_{K})+\frac{\lambda}{\abs{K}}\sum_{\s\in\mcal{E}(K)}F_{\s,K}(\Theta,\uu{v}), \ \forall K \in \mcal{M},
\end{equation*}
with $v_\s = u_\s^n-\frac{\eta\dt}{\veps^2}(\D_\E^{(i)}p(\Theta))_\s$. Note that $H(\lambda,\Theta)$ is a continuous function and thus a homotopy connecting $H(0,\Theta)$ and $H(1,\Theta)$. To prove the existence of a discrete solution $\Theta^{n+1}>0$ of $H(\lambda, \Theta)=0$, it suffices to show that, for any $\lambda\in [0,1]$, the function $H(\lambda,\cdot)$ has a non-zero topological degree with respect to $V$. This is equivalent to showing that $H(\lambda,\cdot)\neq 0$ on $\D V$, for all $\lambda\in [0,1]$. Assume, to the contrary, that $H(\lambda,\Theta)=0$, for some $\lambda\in[0,1]$ and $\Theta\in\D V$, which implies
\begin{equation}
\label{eq:H_zero}
    \frac{1}{\dt}(\Theta_{K}-\rho^n_{K}\theta^n_{K})+\frac{\lambda}{|K|}\sum_{\s\in\mcal{E}(K)}F_{\s,K}(\Theta,\uu{v})=0, \ \forall K \in \mcal{M}.
\end{equation}
Summing \eqref{eq:H_zero} over all $K\in\mcal{M}$, and invoking the definition of $C$ from \eqref{eq:C_thm_ex} yields
\begin{equation}
\label{eq:thetaKub}
    \Theta_K\leqslant \frac{\abs{\Omega}\max_{K\in\mcal{M}}\{\rho_{K}^n\theta^n_{K}\}}{\min_{K\in\mcal{M}}\{\abs{K}\}}< C, \ \forall K\in\mcal{M}.
\end{equation}
Moreover, using \eqref{eq:H_zero}, we find that for each $K\in\mcal{M}$ and for any $\lambda\in[0,1]$,
\begin{equation}
    \label{eq:ept_ex_pos}
    \Theta_{K}\bigg[1+\frac{\lambda\dt}{\abs{K}}\sum_{\substack{\s\in\mcal{E}(K)\\\s=K|L}}|\s|\uu{v}_{\s,K}^+\Bigg]>0.  
\end{equation}
By combining the above inequalities \eqref{eq:thetaKub} and \eqref{eq:ept_ex_pos}, we get that $0<\Theta_K<C$, for each $K\in\mcal{M}$. This contradicts our initial assumption of $\Theta\in\D V$. Therefore, $0\notin H(\{\lambda\}\times\partial V)$ for any $\lambda\in[0,1]$, and consequently, we arrive at $\deg(H(1,\cdot),V,0)=\deg(H(0,\cdot),V,0)$, by using topological degree theory arguments \cite{Dei85}. Furthermore, as the system $H(0,\Theta)$ has one and only one solution, it follows that $\deg(H(0,\cdot),V,0)\neq 0$. Consequently, we conclude that $H(1,\cdot)$ has a zero in $V$, i.e., there exists a solution $\Theta^{n+1}\in L_{\mcal{M}}(\Omega)$ of $H(\lambda, \Theta)=0$ such that $\Theta^{n+1}>0$. Defining $\theta^{n+1}:=\frac{\Theta^{n+1}}{\rho^{n+1}}>0$, we obtain that the triplet $(\rho^{n+1}, \uu{u}^{n+1}, \theta^{n+1})\in L_{\mcal{M}}(\Omega)\times\uu{H}_{\mcal{E},0}(\Omega)\times L_{\mcal{M}}(\Omega)$ is a discrete solution of the system \eqref{eq:ept_dis_update}. This concludes the proof.
\end{proof}

\subsection{Discrete Energy Estimates}
\label{subsec:ept_dis_engy_est}
This section aims to provide discrete equivalents of the energy stability estimates stated in Proposition~\ref{prop:ept_r_engy_balance}. The derivation of these estimates follows the techniques developed in \cite{AGK23} and hence the details are omitted for brevity.
\begin{lemma}[Discrete renormalisation identity]
\label{lem:ept_dis_renorm}
Any solution to the system \eqref{eq:ept_dis_update} satisfies the following equality:
\begin{multline}
\label{eq:ept_dis_renorm}
        \frac{|K|}{\dt}\big[\psi_\gm(\rho_{K}^{n+1}\theta_{K}^{n+1})-\psi_\gm(\rho_{K}^{n}\theta_{K}^{n})\big]+|K|p_{K}^{n+1}(\divM\uu{v}^{n+1})_{K}\\
        +\sum_{\substack{\s\in\mcal{E}(K)\\\s=K|L}}|\s|\Big[\psi_\gm(\rho_{K}^{n+1}\theta_{K}^{n+1})(v_{\s,K}^{n+1})_{+}+\psi_\gm(\rho_{L}^{n+1}\theta_{L}^{n+1})(v_{\s,K}^{n+1})_{-}\Big]+R_{K,\dt}^{n+1}=0.
\end{multline}
Here, the non-negative remainder term $R_{K,\dt}^{n+1}$ is given by
\begin{equation}
\label{eq:ept_RKn+1}
\begin{aligned}
    R_{K,\dt}^{n+1}&=\frac{|K|}{2\dt}\big(\rho_{K}^{n+1}\theta_{K}^{n+1}-\rho_{K}^{n}\theta_{K}^{n}\big)^2\psi_\gm^{\prime\prime}\Big(\bar\Theta_{K}^{n+\frac{1}{2}}\Big)\\
    &\quad+\sum_{\s\in\Ek}|\s|(-(v_{\s,K}^{n+1})_{-})(\rho_{L}^{n+1}\theta_{L}^{n+1}-\rho_{K}^{n+1}\theta_{K}^{n+1})^2\psi_\gm^{\prime\prime}\big(\bar\Theta_{\s}^{n+1}\big),
\end{aligned}
\end{equation}
where $\bar\Theta_{K}^{n+\frac{1}{2}}\in \llbracket\rho_{K}^{n+1}\theta_{K}^{n+1},\rho_{K}^{n}\theta_{K}^{n}\rrbracket$ and $\bar{\Theta}_{\s}^{n+1}\in\llbracket\rho_{K}^{n+1}\theta_{K}^{n+1},\rho_{L}^{n+1}\theta_{L}^{n+1}\rrbracket$. The symbol $\llbracket a, b\rrbracket$ denotes the interval $[\min(a, b), \max(a, b)]$, for any two real numbers $a$ and $b$.
\end{lemma}
\begin{lemma}[Discrete positive renormalisation identity]
\label{lem:ept_dis_porenorm-bal}
Any solution to the system \eqref{eq:ept_dis_update} satisfies the following equality:
\begin{multline}
\label{eq:ept_dis_porenorm}
    \frac{|K|}{\dt}\big[\Pi_\gm(\rho_{K}^{n+1}\theta_{K}^{n+1})-\Pi_\gm(\rho_{K}^{n}\theta_{K}^{n})\big]+|K|p_{K}^{n+1}(\divM\uu{v}^{n+1})_{K}\\
    +\sum_{\substack{\s\in\mcal{E}(K)\\\s=K|L}}|\s|\Big[\big(\psi_\gm(\rho_{K}^{n+1}\theta_{K}^{n+1})-\rho_{K}^{n+1}\theta_{K}^{n+1}\psi_\gm^{\prime}(1)\big)(v_{\s,K}^{n+1})_{+}+\big(\psi_\gm(\rho_{L}^{n+1}\theta_{L}^{n+1})-\rho_{L}^{n+1}\theta_{L}^{n+1}\psi_\gm^{\prime}(1)\big)(v_{\s,K}^{n+1})_{-}\Big]\\
    +R_{K,\dt}^{n+1}=0,
\end{multline}
where the non-negative remainder term is defined in \eqref{eq:ept_RKn+1}.
\end{lemma}

\begin{lemma}[Discrete kinetic energy identity]
\label{lem:ept_dis_kinbal}
Any solution to the system \eqref{eq:ept_dis_update} satisfies the following equality for $1\leqs i\leqs d,\; \s\in\E_\intr^{(i)}$ and $0\leqs n\leqs{N-1}$:
\begin{equation}
\label{eq:ept_dis_kinbal}
\begin{aligned}
 \frac{1}{2}\frac{|\Ds|}{\dt}\big(\rho_{\Ds}^{n+1}(u_\s^{n+1})^2-\rho_{\Ds}^{n}(u_\s^{n})^2\big)+\sum_{\eps\in\tilde\E(\Ds)}F_{\eps,\s}(\rho^{n},\uu{u}^{n})\frac{|u_\eps^{n}|^2}{2}+\frac{1}{\veps^2}|\Ds|v_\s^{n+1}\big(\D^{(i)}_\E p^{n+1}\big)_{\s}+R_{\s,\dt}^{n+1}\\
 =-|\Ds|\delta u_{\s}^{n+1}\frac{1}{\veps^2}\big(\D^{(i)}_\E p^{n+1}\big)_\s,
\end{aligned}
\end{equation}
where the remainder term $R_{\s,\dt}^{n+1}$ is defined by
\begin{equation}
    \label{eq:ept_Rsn+1}
    R_{\s,\dt}^{n+1}=
    -\frac{|\Ds|}{2\dt}\rho_{\Ds}^{n+1}(u_\s^{n+1}-u_\s^n)^2-\sum_{\substack{\eps\in\tilde\E(\Ds)\\\eps=\Ds|D_{\s^{\prime}}}}F_{\eps,\s}(\rho^{n},\uu{u}^{n})^{-}\frac{(u_\s^n-u_{\s^{\prime}}^n)^2}{2}.
\end{equation}
\end{lemma}

\begin{theorem}[Total energy estimate]
\label{thm:ept_dis_totbal}
    Any solution to the system \eqref{eq:ept_dis_update} satisfies the following energy inequality, for each $0\leqs n\leqs{N-1}$:
    \begin{equation}
        \label{eq:ept_dis_totbal}
        \begin{split}
             \frac{1}{\veps^2}\sum_{K\in\M}|K|\Pi_\gm(\rho_{K}^{n+1}\theta_{K}^{n+1})+\sum_{\s\in\E_\intr}|\Ds|\frac{1}{2}\rho_{\Ds}^{n+1}(u_{\s}^{n+1})^2 \leqs \frac{1}{\veps^2}\sum_{K\in\M}|K|\Pi_\gm(\rho_{K}^{n}\theta_{K}^{n})\\
             +\sum_{\s\in\E_\intr}|\Ds|\frac{1}{2}\rho_{\Ds}^{n}(u_{\s}^{n})^2,
        \end{split}
    \end{equation}
    under the following conditions:    
    \begin{enumerate}[label=(\roman*)]
        \item a lower bound on the stabilisation parameter:
        \begin{equation}
        \label{eq:eta_lb}
            \displaystyle \eta \geqs\frac{3}{2\rho_{\Ds}^n},\; \forall\s \in\E_\intr^{(i)}, \ 1 \leqs i \leqs d,
        \end{equation}
        \item  a CFL restriction on the timestep:
        \begin{equation}
            \label{eq:ept_cfl_stab}
            \frac{\dt}{\abs{K}}\sum\limits_{\s\in\Ek}\frac{|F_{\s,K}(\rho^n,\uu{u}^n)|}{\rho_K^n}\leqslant \frac{1}{3},\;\forall K\in\M.
        \end{equation}
    \end{enumerate}
\end{theorem}
\begin{proof}
    We take sum over $K\in\M$ in \eqref{eq:ept_dis_porenorm}, multiply by $\frac{1}{\veps^2}$, and over $\s\in\E_\intr$ in \eqref{eq:ept_dis_kinbal}, and add the resulting equations to get
\begin{equation}
    \label{eq:ept_thm_TE_eq1}
    \begin{aligned}
    \frac{1}{\veps^2\dt}\sum_{K\in\M}|K|\Big(\Pi_\gm(\rho_{K}^{n+1}\theta_{K}^{n+1})-\Pi_\gm(\rho_{K}^{n}\theta_{K}^{n})\Big)+\sum_{\s\in\E_\intr}\frac{1}{2}\frac{|\Ds|}{\dt}\Big(\rho_{\Ds}^{n+1}(u_\s^{n+1})^2-\rho_{\Ds}^{n}(u_\s^n)^2\Big)\\
    \quad+\mcal{R}_{\M,\dt}^{n+1}+\mcal{R}_{\E,\dt}^{n+1}=-\frac{\eta\dt}{\veps^4}\sum_{i=1}^d\sum_{\s\in\E_\intr^{(i)}}|\Ds|\big(\D^{(i)}_\E p^{n+1}\big)_\s^2,
    \end{aligned}
\end{equation}
where the global remainder terms $\mcal{R}_{\M,\dt}$ and $\mcal{R}_{\E,\dt}$ are obtained by summing the local remainders $R_{K,\dt}$ and $R_{\s,\dt}$, respectively. Clearly, $\mcal{R}_{\M,\dt}$ is non-negative unconditionally. We estimate the remainder term $R_{\E,\dt}$ in terms of flux contributions and pressure gradient along the lines of \cite[Theorem 4.5]{AGK23} and rewrite \eqref{eq:ept_thm_TE_eq1} as
\begin{equation}
\label{eq:ept_dis_totenbal_press}
\begin{aligned}
   &\frac{1}{\veps^2\dt}\sum_{K\in\M}|K|\Big(\Pi_\gm(\rho_{K}^{n+1}\theta_{K}^{n+1})-\Pi_\gm(\rho_{K}^{n}\theta_{K}^{n})\Big)+\sum_{\s\in\E_\intr}\frac{1}{2}\frac{|\Ds|}{\dt}\Big(\rho_{\Ds}^{n+1}(u_\s^{n+1})^2-\rho_{\Ds}^{n}(u_\s^n)^2\Big)\\
    &+\frac{\dt}{\veps^4}\sum_{i=1}^d\sum_{\s\in\E_\intr^{(i)}}\bigg(\eta- \frac{1}{\rho_{\Ds}^{n+1}}\bigg)\big(\D_\E^{(i)}p^{n+1}\big)_\s^2\\
   &\leqs\sum_{\s\in\E_\intr}\left(\frac{1}{2}-\frac{\dt}{|\Ds|}\sum_{\eps\in\tilde\E(\Ds)}\frac{-(F_{\eps,\s}(\rho^{n},\uu{u}^{n}))^{-}}{\rho_{\Ds}^{n+1}}\right)\left(\sum_{\substack{\eps\in\bar\E(\Ds)\\\eps=\Ds|D_{\s^{\prime}}}}F_{\eps,\s}(\rho^{n},\uu{u}^{n})^-(u_\s^{n}-u_{\s^\prime}^{n})^2\right).
\end{aligned}
\end{equation}
The right-hand side is non-positive under the time step condition \eqref{eq:ept_cfl_stab}, while the third term on the left-hand side is non-negative, provided that the condition \eqref{eq:eta_lb} is satisfied. This completes the proof.
\end{proof}

\section{Weak Consistency of the Scheme}
\label{sec:ept_cons_LW}
In this section, we establish the weak consistency of the proposed scheme \eqref{eq:ept_dis_update} in the sense of Lax-Wendroff. We show that if a sequence of approximate solutions generated by successive mesh refinements remains bounded under suitable norms and converges strongly, then the limit must be a weak solution of the system \eqref{eq:euler_pot-temp}. 

\begin{theorem}
\label{thm:ept_weak_cons}
Assume that
$\big(\M^{(m)},\E^{(m)})_{m\in\mbb{N}}$ is a sequence of MAC grids and 
$\delta t^{(m)}$ is a sequence of timesteps such that both
$\lim_{m\rightarrow \infty}\delta t^{(m)}$ and $\lim_{m\rightarrow
  \infty}h^{(m)}$ are $0$, where $h^{(m)} =
\max_{K\in\M^{(m)}}\diam(K)$. Let
$\big(\rho^{(m)},\uu{u}^{(m)},\theta^{(m)}\big)_{m\in\mbb{N}}$ be the corresponding sequence of discrete solutions for an initial datum $(\rho_0,\uu{u}_0,\theta_{0})\in L^\infty(\Omega)^{d+2}$. We
assume that $(\rho^{(m)},\uu{u}^{(m)}, \theta^{(m)})_{m\in\mbb{N}}$
satisfies the following. 
\begin{enumerate}[label=(\roman*)]
\item $\big(\rho^{(m)},\uu{u}^{(m)},\theta^{(m)}\big)_{m\in\mbb{N}}$ is
  uniformly bounded in $L^\infty(Q)^{1+d}$, i.e.\ 
\begin{subequations}
\label{eq:solu_abs_bound}
\begin{align}
\label{eq:ht_abs_bound}
\ubar{C}<(\rho^{(m)})^n_K \leqslant \bar{C}, \ \forall
  K\in\mcal{M}^{(m)}, \ 0\leqslant n\leqslant N^{(m)}, \ \forall
  m\in\mbb{N}, 
\end{align}
\begin{align}
\label{eq:ept_temp_abs_bound}
  \ubar{C}<(\theta^{(m)})^n_K &\leqslant \bar{C}, \ \forall
                                K\in\mcal{M}^{(m)}, \ 0\leqslant
                                n\leqslant N^{(m)}, \ \forall
                                m\in\mbb{N},
\end{align}
\begin{align}
  \label{eq:ept_u_abs_bound}
  |(u^{(m)})^n_\s| &\leqslant C, \ \forall
                         \s\in\mcal{E}^{(m)}, \ 0\leqslant
                         n\leqslant N^{(m)}, \ \forall m\in\mbb{N}.  
\end{align}
\end{subequations}
where $\ubar{C}, \bar{C}, C>0$ are constants independent of the
discretisations.  
\item $\big(\rho^{(m)},\uu{u}^{(m)}, \theta^{(m)}\big)_{m\in\mbb{N}}$ converges to $(\rho,\uu{u}, \theta)\in L^\infty(0,
  T;L^\infty(\Omega)^{1+d+1})$ in $L^r(Q_T)^{1+d+1}$ for $1\leqslant
  r<\infty$.  
\end{enumerate}
Furthermore, assume that the sequence of grids
$\big(\M^{(m)},\E^{(m)}\big)_{m\in\mbb{N}}$ and the timesteps $\delta 
t^{(m)}$ satisfies the regularity conditions: 
\begin{equation}
\label{eq:CFL_restric}
\frac{\delta t^{(m)}}{\min_{K\in\mcal{M}^{(m)}}\{\diam(K)\}}\leqslant\mu,\
\max_{K \in \mcal{M}^{(m)}}
\frac{\diam(K)^2}{\abs{K}}\leqslant\mu,\;\forall m\in\mbb{N}, 
\end{equation}
where $\mu>0$ is independent of the discretisations. Then
$(\rho,\uu{u},\theta)$ is a weak solution to
\eqref{eq:euler_pot-temp}.
\end{theorem}
\begin{proof}
     The proof proceeds along the same lines as in \cite{AGK23}, with analogous calculations in \cite{GHL22,HLN+23}, and hence we omit the details. 
\end{proof}

\section{Consistency of the Scheme with the Zero Mach Number Limit}
\label{sec:ept_cons_incomp}
The present section aims to derive the zero Mach number limit of the semi-implicit scheme \eqref{eq:ept_dis_update} as done in the continuous case in Section~\ref{sec:ept_cont-case}. To begin with, we combine the discrete total energy estimate \eqref{eq:ept_dis_totbal}, the assumption \eqref{def:ept_ill_prep_id} on ill-prepared initial data, and the bound on $\Pi_\gm$ from \cite[Lemma 2.3]{HLS21} to obtain the following lemma.
\begin{lemma}
\label{lem:ept_dis_ent}
Suppose that the initial data $(\rho_{0}^{\veps},\uu{u}_{0}^{\veps}, \theta_{0}^{\veps})$ are ill-prepared in the sense of Definition~\ref{def:ept_ill_prep_id}. Then there exists a constant $C>0$, independent of $\veps$, such that the numerical solution $(\rho^n,\uu{u}^n, \theta^n)_{0\leqs n\leqs N}$ satisfies the entropy inequality:
\begin{equation}
\label{eq:ept_bdd_terms}
    \frac{1}{\veps^2}\sum_{K\in\mcal{M}}|K|\Pi_{\gm}(\rho_{K}^{n}\theta_{K}^{n})+\frac{1}{2}\sum_{\s\in\mcal{E}_\intr}|\Ds|\rho_{\Ds}^{n}|\uu{u}_{\s}^{n}|^2+\frac{(\dt)^2}{\veps^4}\sum_{k=0}^{n-1}\sum_{\s\in\mcal{E}_\intr}\left(\eta-\frac{1}{\rho_{\Ds}^{k+1}}\right)|(\grd_{\mcal{E}}p^{k+1})_{\s}|^2\leqs C.
\end{equation}
\end{lemma}

The analysis carried out subsequently depends on the assumption that the sequence of approximate solutions for the density is uniformly bounded, i.e.,
\begin{equation}
    \label{eq:den_bdd}
    0<\ubar{c}<(\rho^{\veps})^n_K<\bar{C}, \quad \forall K\in\M \text{ and } 0\leqs n\leqs N,
\end{equation}
where the constants $\ubar{C}$ and $\bar{C}$ depend on the mesh and the time-step, but not on $\veps$.

In the following, we state a few claims similar to the ones in \cite[Proposition 2.6]{HLS21} for the discrete solution.
\begin{lemma}
\label{lem:ept_disc_conv_bdd}
Let $\mcal{T}=(\M,\E)$ be a fixed primal-dual mesh pair that gives a MAC discretisation of $\bar{\Omega}$ and let $\delta t>0$ be such that $\{0=t^0<t^1<\cdots<t^N=T\}$ is a discretisation of the time domain $[0,T]$ with  $t^n = n\dt$, for $1\leqs n\leqs N = \lfloor\frac{T}{\delta t}\rfloor$. For each $\veps>0$, let us denote by $(\rho^\veps,\uu{u}^\veps, \theta^\veps)\in L^\infty(0,T;\Lm\times\uu{H}_{\mcal{E},0}(\Omega)\times \Lm)$, the discrete solution to the semi-implicit scheme \eqref{eq:ept_dis_update} with respect to the ill-prepared initial data discretised via \eqref{eq:ept_dis_ic}. Then the following holds.
\begin{enumerate}[label=(\roman*)]
\item For $\gm\geqs 2$, we have for all $\veps>0$, sufficiently small,
\begin{equation}
\label{eq:ept_disc_dens_est_gamma_big}
\frac{1}{\veps}\norm{\rho^{\veps}\theta^{\veps}-1}_{L^\infty(0,T;L^2(\Omega))}\leqs C(\gm).
\end{equation}
For $\gm\in(1,2)$, we have for all $\veps>0$, sufficiently small, and for any $R\in(2,+\infty)$,
\begin{align}
\frac{1}{\veps}\norm{(\rho^{\veps}\theta^{\veps}-1)\mcal{X}_{\{\rho^{\veps}\theta^{\veps}<R\}}}_{L^\infty(0,T;L^2(\Omega))}&\leqs C(\gm,R),\label{eq:ept_disc_dens_est_gamma_ess} \\
{\veps}^{-\frac{2}{\gm}}\norm{(\rho^{\veps}\theta^{\veps}-1)\mcal{X}_{\{\rho^\veps\theta^{\veps}\geqs R\}}}_{L^\infty(0,T;L^\gm(\Omega))}&\leqs C(\gm,R).\label{eq:ept_disc_dens_est_gamma_res}
\end{align}
Here $C(\gm)$ and $C(\gm,R)$ are positive constants independent of $\veps$.
\item $\rho^\veps\theta^{\veps}\rightarrow 1$ as $\veps\rightarrow 0$ in $L^\infty(0,T;L^r(\Omega))$ for any $r\in(1,\min\{2,\gm\}]$. 
\item The sequence of approximate solutions for the velocity component $(\uu{u}^\veps)_{\veps>0}$ is uniformly bounded with respect to $\veps$, i.e.,\ for all $\veps>0$, sufficiently small, we have
\begin{equation}
\label{eq:ept_disc_vel_bound}
\norm{\uu{u}^\veps}_{L^\infty(0,T;L^2(\Omega)^d)}\leqs C,
\end{equation}
where $C>0$ is a constant independent of $\veps$.
\end{enumerate}
\end{lemma}

We also have the following estimate on the second-order pressure term,  which is a consequence of the inf-sup stability of the MAC discretisation \cite{HLS21} and the uniform bound on the discrete pressure gradient obtained from estimate \eqref{eq:ept_bdd_terms}. 
\begin{lemma}[Bound on the second-order pressure]
\label{lem:ept_bdd_press}
Let $\veps>0$ and $(\rho^{n},\uu{u}^{n}, \theta^n)_{0\leqs n\leqs N}$ be the solution to the scheme \eqref{eq:ept_dis_update}. Define $\pi^n=\sum_{K\in\M}\pi^n_K\mcal{X}_{K}$, where $\pi^n_K=\frac{p^{n}_K-m(p^n)}{\veps^2}$ with $m(p^n)$ being the mean value of $p^n$ over $\Omega$. Then there exists a constant $C_{\mcal{T},\dt}$, independent of $\veps$, such that for $0\leqs n\leqs N$
\begin{equation}
\label{eq:ept_bdd_press}
    \norm{\pi^n}\leqs C_{\mcal{T},\dt},
\end{equation}
where $\norm{\cdot}$ is any norm on the discrete function space.
\end{lemma}

We now state the main result of this section: the convergence, up to a subsequence, of the numerical solutions toward the solution of a semi-implicit incompressible scheme. The uniform estimates established above enable us to pass to the limit as 
$\veps \to 0$.
\begin{theorem}[Asymptotic Limit of the Scheme]
Let $(\veps^{(k)})_{k\in \mbb{N}}$ be a sequence of positive numbers converging to zero, $(\rho^{(k)},\uu{u}^{(k)}, \theta^{(k)})_{k\in\mbb{N}}$ be the corresponding sequence of numerical solutions obtained from the scheme \eqref{eq:ept_dis_update}, and let the initial
data $(\rho^{(k)}_{0},\uu{u}^{(k)}_{0},\theta^{(k)}_0)$ satisfy
\eqref{eq:ept_ill_prep_id}. Then $(\rho^{(k)}\theta^{(k)})_{k\in\mbb{N}}$
converges to $1$ in $L^{\infty}(0,T;L^{\gm}(\Omega))$ and $(\rho^{(k)},\uu{u}^{(k)},\pi^{(k)})_{k\in\mbb{N}}$ converges to $(\rho, \uu{U},\pi)\in L^\infty(0,T; L_\M(\Omega)\times\uu{H}_{\E,0}(\Omega)\times L_\M(\Omega))$ in any discrete norm when $k$ tends to $\infty$, where
the sequence $(\rho^n,\uu{U}^n,\pi^n)$ is defined as follows. Given
$(\rho^n, \uu{U}^n, \pi^n)\in L_\M(\Omega)\times\uu{H}_{\E,0}(\Omega)\times L_\M(\Omega)$ at
time $t^n$, $(\rho^{n+1}, \uu{U}^{n+1}, \pi^{n+1})\in L_\M(\Omega)\times\uu{H}_{\E,0}(\Omega)\times
L_\M(\Omega)$ is obtained as the solution of the following semi-implicit scheme: 
\begin{subequations}
\label{eq:ept_dis_update_incomp}
    \begin{gather}
        \frac{1}{\dt}\big(\rho_{K}^{n+1}-\rho_{K}^{n}\big)+\frac{1}{\left|K\right|}\sum_{\s\in\E(K)}F_{\s,K}(\rho^{n},\uu{U}^{n})=0,\; \forall K \in \M, \label{eq:ept_dis_incomp_mas}\\
    \frac{1}{\dt}\big(\rho_{\Ds}^{n+1}U_\s^{n+1}-\rho_{\Ds}^{n}U_\s^{n}\big)+\frac{1}{\left|\Ds\right|}\sum_{\eps\in\tilde{\E}(\Ds)}F_{\eps,\s}(\rho^n,\uu{U}^{n})U_{\eps,\mathrm{up}}^{n}+(\partial^{(i)}_{\E}\pi^{n+1})_{\s}=0, \ 1\leqs i\leqs d, \ \forall \s\in\E_\intr^{(i)}, \label{eq:ept_dis_incomp_mom}\\
    (\dive_\M(\uu{U}^n-\delta\uu{U}^{n+1}))_{K}=0,\; \forall K \in \M, \label{eq:ept_dis_incomp_temp}
\end{gather}
\end{subequations}
where the correction $\delta U^{n+1}_\s=\eta\dt(\D_\E^{(i)}\pi^{n+1})_\s$ for $\s\in\Eint^{(i)}$, $i=1,2,\dots d$.
\end{theorem}

We now show that the limiting scheme \eqref{eq:ept_dis_update_incomp} is energy stable by establishing a total energy estimate.
\begin{theorem}[Total energy estimate]
\label{thm:dis_totbal_incom}
Any solution to the system \eqref{eq:ept_dis_update_incomp} satisfies the following entropy inequality, for each $0\leqs n \leqs N$: 
\begin{equation}
\label{eq:dis_totbal_incom}
 \sum_{\s\in\E_\intr}|\Ds|\frac{1}{2}\rho_{\Ds}^{n+1}(U_{\s}^{n+1})^2 \leqs 
  \sum_{\s\in\E_\intr}|\Ds|\frac{1}{2}\rho_{\Ds}^{n}(U_{\s}^{n})^2,
\end{equation}
 under the following conditions $\forall\s \in\E_\intr^{(i)}, \ 1 \leqs i \leqs d$:
\begin{enumerate}[label=(\roman*)]
    \item a lower bound on the stabilisation parameter:
    \begin{equation}
    \label{eq:eta_lb_incom}
        \displaystyle \eta \geqs\frac{3}{2\rho_{\Ds}^n},
    \end{equation}
    \item a CFL restriction on the timestep:
    \begin{equation}
    \label{eq:cfl_stab_incomp}
       \frac{\dt}{\abs{K}}\sum\limits_{\s\in\Ek}\frac{|F_{\s,K}(\rho^n,\uu{U}^n)|}{\rho_K^n}\leqslant \frac{1}{3}.
    \end{equation}
\end{enumerate}
\end{theorem}
\begin{proof}
We multiply the momentum balance \eqref{eq:ept_dis_incomp_mom} by $\abs{\Ds}U_\s^n$ and take the sum over all $\s\in\E_\intr$ to get
\begin{equation}
    \label{eq:thm_eq}
    \begin{aligned}
    \sum_{\s\in\E_\intr}\frac{1}{2}\frac{|\Ds|}{\dt}\Big(\rho_{\Ds}^{n+1}(U_\s^{n+1})^2-\rho_{\Ds}^{n}(U_\s^n)^2\Big)+\sum_{\s\in\Eint}|\Ds|(U_{\s}^n-\delta U_\s^{n+1})\big(\D^{(i)}_\E \pi^{n+1}\big)_{\s}\\
    =\sum_{\s\in\Eint}\mcal{R}_{\s,\dt}^{n+1}-\sum_{\s\in\Eint}|\Ds| \delta U_{\s}^{n+1}\big(\D^{(i)}_\E \pi^{n+1}\big)_\s,
    \end{aligned}
\end{equation}
where
\begin{equation}
    \label{eq:tilde_Rsn+1}
    R_{\s,\dt}^{n+1}=
    \frac{|\Ds|}{2\dt}\rho_{\Ds}^{n+1}(U_\s^{n+1}-U_\s^n)^2+\sum_{\substack{\eps\in\tilde\E(\Ds)\\\eps=\Ds|D_{\s^{\prime}}}}F_{\eps,\s}(\rho^{n},\uu{U}^{n})^{-}\frac{(U_\s^n-U_{\s^{\prime}}^n)^2}{2}.
\end{equation}
Analogous to the proof of Theorem~\ref{thm:ept_dis_totbal}, we obtain the following inequality:
\begin{equation}
    \label{eq:thm_eq_last}
    \begin{aligned}
    \sum_{\s\in\E_\intr}\frac{1}{2}\frac{|\Ds|}{\dt}\Big(\rho_{\Ds}^{n+1}(U_\s^{n+1})^2-\rho_{\Ds}^{n}(U_\s^n)^2\Big)
    \leqs\dt\sum_{\s\in\Eint}\bigg(\frac{1}{\rho^{n+1}_{\Ds}}-\eta\bigg)\abs{\Ds}(\D_\E^{(i)}\pi^{n+1})_\s^2\leqs 0,
    \end{aligned}
\end{equation} 
using the stability conditions \eqref{eq:eta_lb_incom}-\eqref{eq:cfl_stab_incomp} and the divergence-free condition \eqref{eq:ept_dis_incomp_temp}.
\end{proof}

\section{Numerical Results}
\label{sec:ept_num_res}
This section presents a numerical evaluation of the proposed scheme. The numerical implementation of the scheme \eqref{eq:euler_pot-temp} requires the time-step condition \eqref{eq:ept_cfl_stab} to be imposed, as shown in the stability analysis in Theorem~\ref{thm:ept_dis_totbal}. However, this timestep condition contains flux terms, making its implementation difficult. To overcome this difficulty, we derive a sufficient condition that is easier to enforce computationally.
\begin{proposition}
Suppose $\dt>0$ is such that for each $\s= K|L\in\E_\intr^{(i)},\; i\in\{1,\dots,d\}$, the following holds: 
\begin{equation}
\label{eq:time-step_suff}
    \dt\max\Bigg\{\frac{\abs{\D K}}{\abs{K}},\frac{\abs{\D L}}{\abs{L}}\Bigg\}\abs{u^n_\s} \leqs \frac{\beta}{1+\beta}\displaystyle\frac{\min\{\rho^n_K,\rho^n_L\}}{\max\{\rho^n_K,\rho^n_L\}},
\end{equation}
where $\abs{\D K}=\sum\limits_{\s\in\E(K)}\abs{\s}$ and $\beta\leqs \frac{1}{2}$. Then $\dt$ satisfies the time-step restriction \eqref{eq:ept_cfl_stab}.
\end{proposition}
\begin{proof}
    The proof follows the approach of \cite[Proposition 3.2]{CDV17}, with analogous arguments also appearing in \cite{AGK23, DVB17}.
\end{proof}

\subsection{Two Colliding Acoustic Pulses}
\label{subsec:ept_coll_pulse}
We consider the evolution of two colliding acoustic pulses as in \cite{Kle95} with the following ill-prepared initial data: 
\begin{align*}
    \rho(0, x) &= 0.955 +\frac{\veps}{2}(1-\cos (2\pi x)),
    \\
    u(0, x) &= -\sign(x)\sqrt{\gm}(1-\cos (2\pi x)),\\
    (\rho\theta)^\gamma(0, x) &= 1.0 + \veps \gm (1 - \cos (2\pi x)),
\end{align*}
where $\gm=1.4$. The computational domain $[-1,1]$ is divided into $200$ mesh points. 
\begin{figure}[htbp]
    \centering
    \includegraphics[height=0.5\textheight]{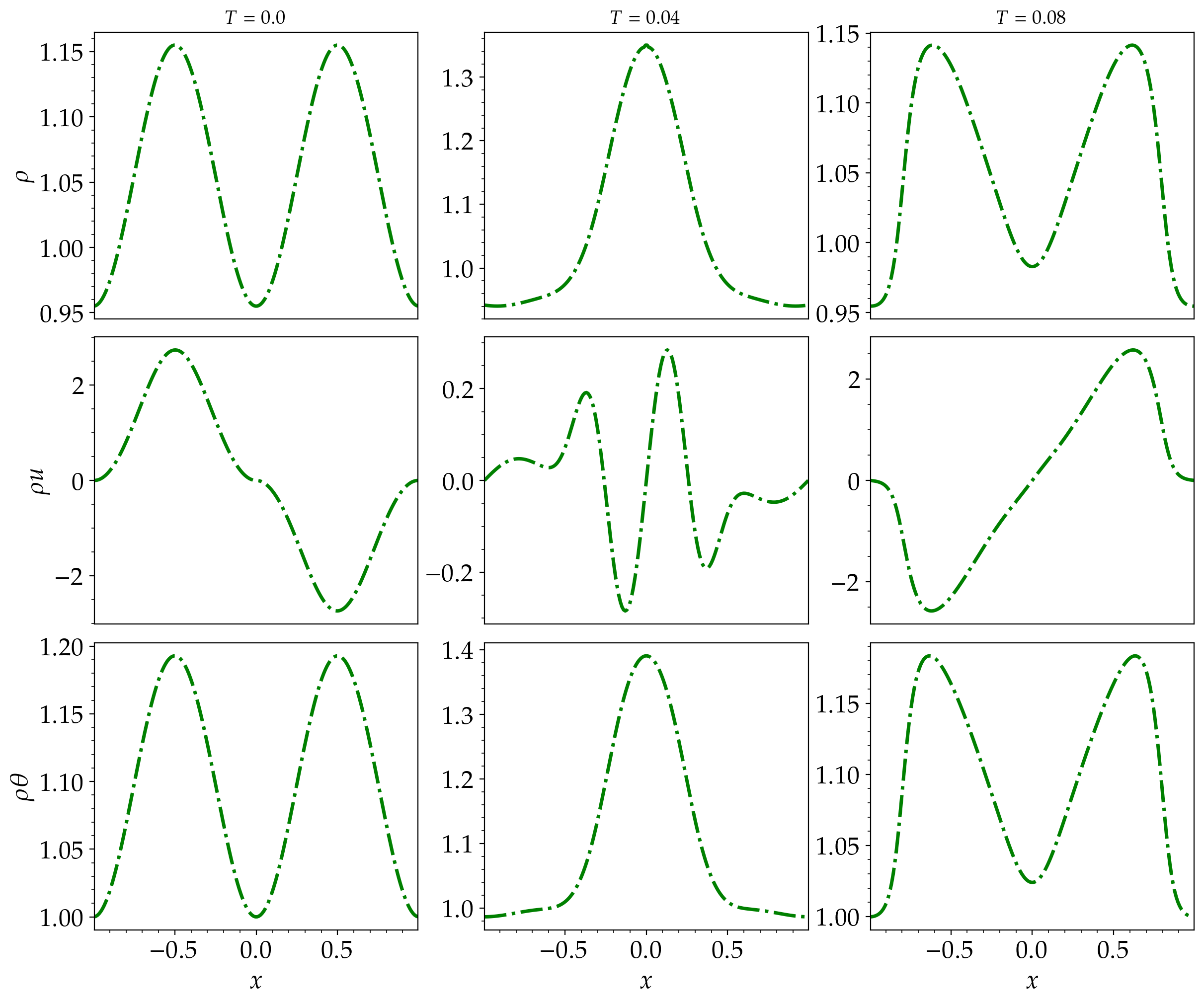}  
    \caption{Plots of the density, the momentum, and the total potential temperature for the colliding pulses problem corresponding to $\veps=0.1$.}
    \label{fig:ept_coll-pulse_eps-01}
\end{figure}
The initial data describe two acoustic pulses: one propagating to the right and the other to the left. As shown in Figure~\ref{fig:ept_coll-pulse_eps-01}, the two pulses superimpose and then separate over time—a process that repeats due to the periodic boundary conditions. The pulses overlap at $T = 0.04$, produce maximum amplitude, and by $T = 0.08$, they separate again and begin to steepen as a result of weakly nonlinear effects. Notably, the scheme preserves the amplitude of the pulses, even in the presence of low Mach number flow and ill-prepared initial data.

\subsection{Positivity Test}
\label{subsec:pos_test}
To verify the positivity-preserving property of the scheme, we perform the simulation of an extreme Riemann problem. The initial data read 
\begin{equation*}
      (\rho,u,\theta)=
      \begin{cases}
        (1, -2, 0.52), & \mbox{if} \ x<0.5,\\
        (1, 2, 0.52), & \mbox{if} \ x>0.5.
      \end{cases}
\end{equation*}
The computational domain $[0,1]$ is uniformly discretised using $100$ mesh points, and the simulation is carried out up to a final time $T = 0.15$. The adiabatic index is set to $\gamma = 1.4$, and periodic boundary conditions are imposed. We set $\veps = 1$ to model the compressible regime. Figure~\ref{fig:ex_riemann} shows the computed density, momentum, and total potential temperature. The results demonstrate the scheme’s strong ability to preserve the positivity of both density and temperature.
\begin{figure}[htbp]
    \centering
    \includegraphics[height=0.204\textheight]{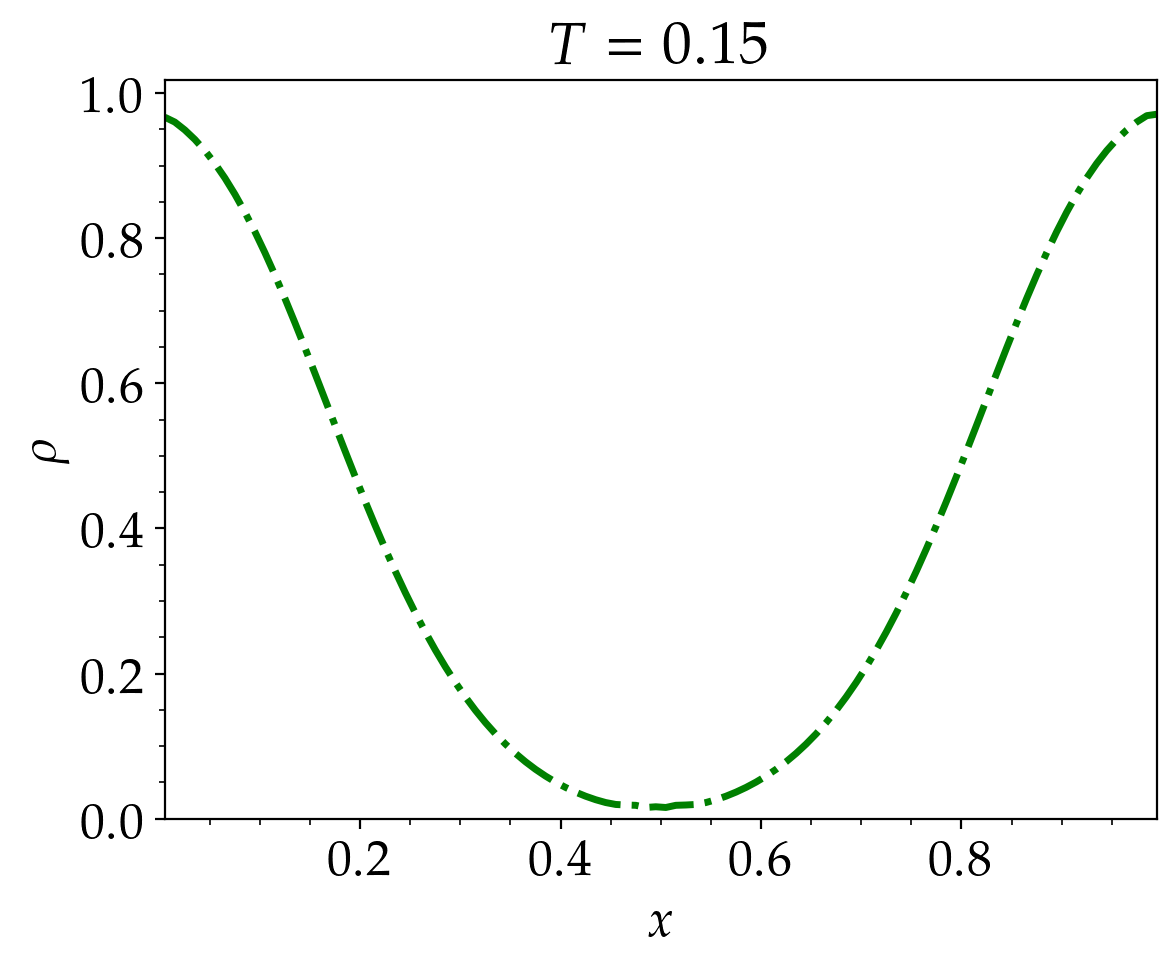}  
    \includegraphics[height=0.204\textheight]{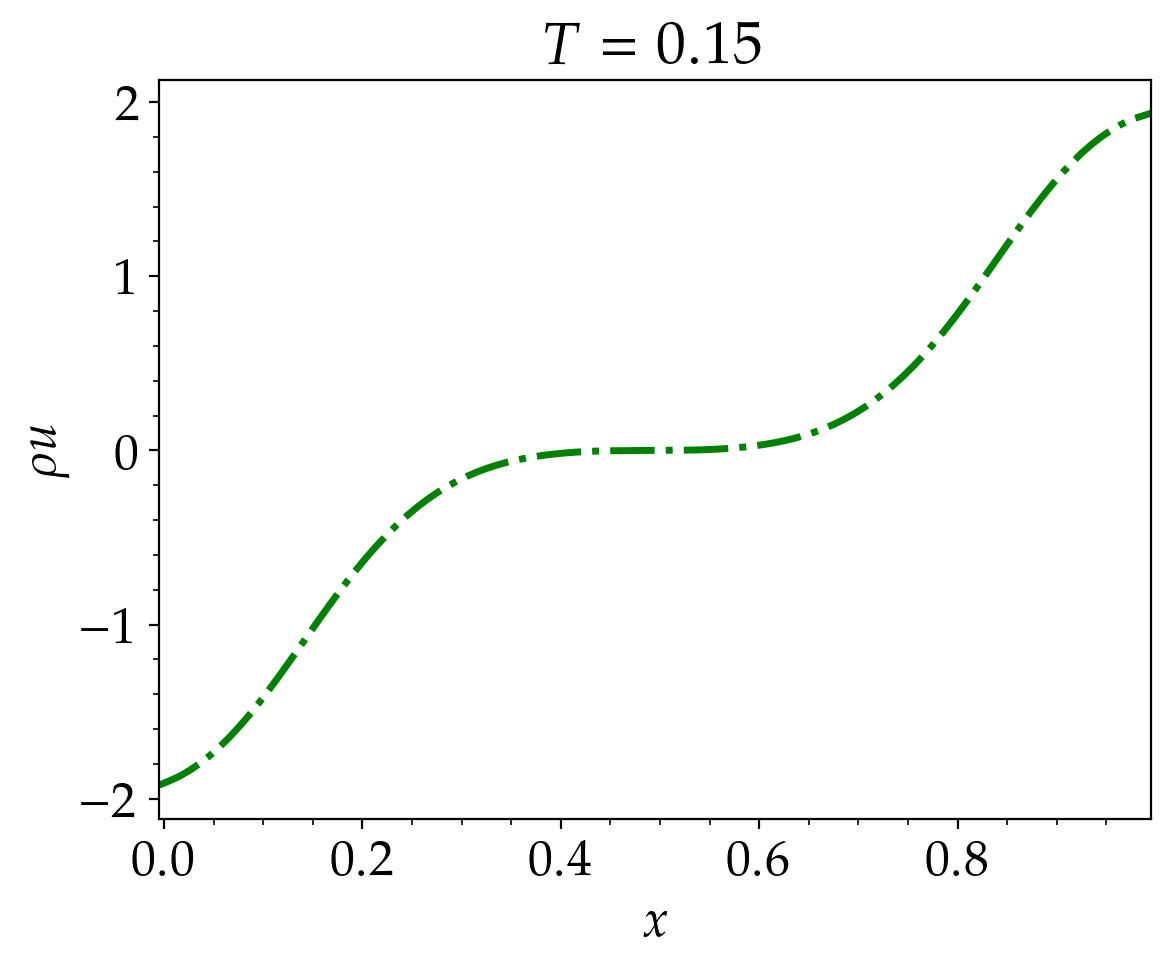}
    \includegraphics[height=0.204\textheight]{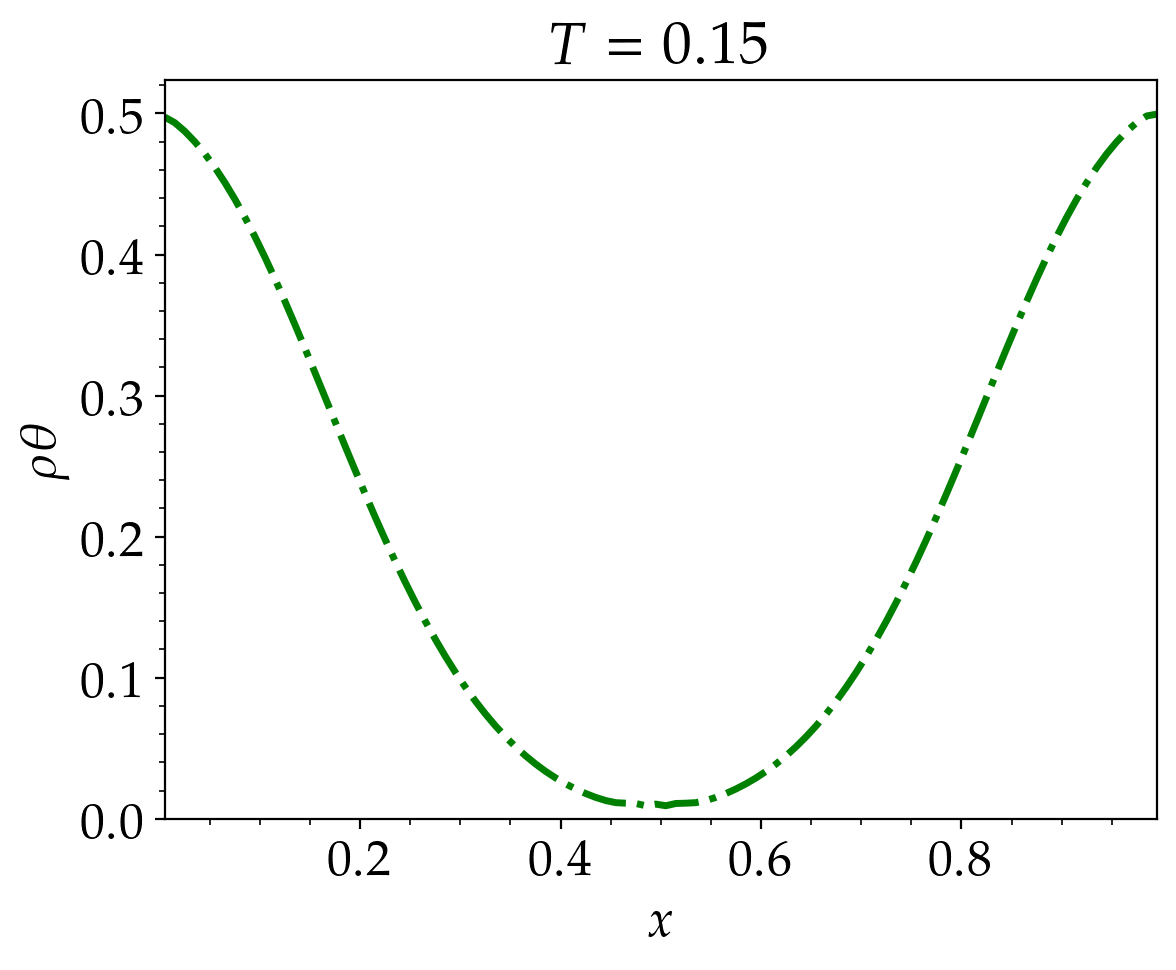}
    \caption{Plots of the density, the momentum, and the total potential temperature for the extreme Riemann problem at time $T=0.15$.}
    \label{fig:ex_riemann}
\end{figure}

\subsection{1D Riemann Problem}
\label{sec:1d_riemann}
The motivation for this test problem is drawn from \cite{DLV17}. In the domain $[0,1]$, the initial density and temperature are uniform and set to $1$. The initial velocity is defined by
\begin{equation}
    u(0, x) = 
\begin{cases}
1 - \frac{\veps^2}{2}, & \text{if } x \leqs 0.2 \text{ or } x \geqs 0.8, \\
1 + \frac{\veps^2}{2}, & \text{if } 0.25 \leqs x \leqs 0.75, \\
1, & \text{otherwise}.
\end{cases}
\end{equation}
The computational domain is uniformly discretised using 300 mesh points, with periodic boundary conditions imposed. The adiabatic index is set to $\gamma = 1.4$, and simulations are carried out up to a final time $T = 0.05$ for two different values of $\varepsilon$. Figure~\ref{fig:ept_riemann1d_soln} presents the profiles of density, momentum, and potential temperature, compared against a reference solution computed using an explicit Rusanov scheme with 10,000 mesh points. The top row of the figure corresponds to $\varepsilon = 1$, demonstrating that the proposed AP scheme accurately captures both shock and expansion regions. In contrast, the bottom row illustrates the results for $\varepsilon = 0.01$, where the AP scheme successfully captures the incompressible limit. The computed flow variables remain nearly constant, aligning well with theoretical expectations.
\begin{figure}[htbp]
    \centering
    \includegraphics[height=0.35\textheight]{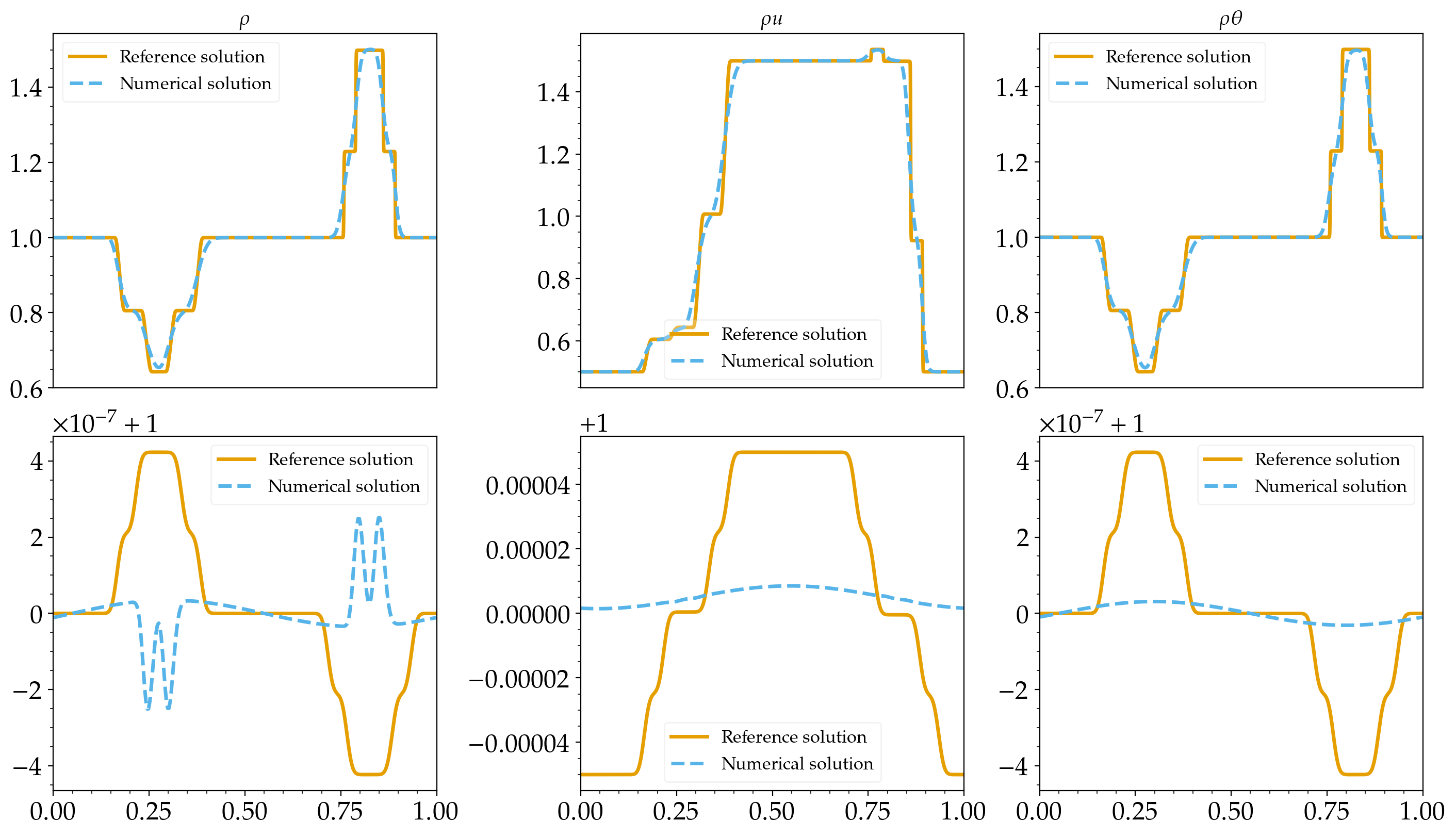}
    \caption{Density, momentum, and the total potential temperature profiles for the 1d Riemann problem at $T = 0.05$. The top row corresponds to $\veps = 1$, illustrating the compressible regime, while the bottom row corresponds to $\veps = 0.01$, highlighting the performance of the AP scheme in the incompressible limit.}
     \label{fig:ept_riemann1d_soln}
\end{figure}

\subsection{Stationary Vortex}
\label{subsec:ept_stationary_vrtx}
Drawing inspiration from \cite{Zen18}, we consider a stationary vortex problem as a benchmark to evaluate the performance of the numerical scheme in low Mach number flow regimes, particularly focusing on how its numerical dissipation depends on the Mach number.

A radially symmetric vortex is placed at the point $(x_c,y_c)=(0.5,0.5)$ in the computational domain $\Omega = [0,1]\times[0,1]$, and the initial data read
\begin{equation}
    \begin{aligned}
        \rho(0, x, y) &= 1 + \frac{\veps^2}{2} \int_0^r \frac{u_\vartheta^2(s)}{s} \, \dd s, \ \theta(0, x, y) = 1, \\
   u(0, x, y) &= u_\vartheta(r) \frac{y - y_c}{r}, \ v(0, x, y) = -u_\vartheta(r) \frac{x - x_c}{r},
    \end{aligned}
\end{equation}
where $r=\sqrt{(x-x_c)^2+(y-y_c)^2}$ and the angular velocity $u_\vartheta(r)$ is defined as
\begin{equation*}
    u_\vartheta(r) =
\begin{cases}
a_1 r, & \text{if } r \leqs r_1, \\
a_2 + a_3 r, & \text{if } r_1 < r \leqs r_2, \\
0, & \text{otherwise}.
\end{cases}
\end{equation*}
Here, $r_1 = 0.2$ and $r_2 = 0.4$ are the inner and outer radii and the constants are chosen as
\begin{equation*}
    a_1 = \frac{a}{r_1}, \quad
    a_2 = -\frac{a r_2}{r_1 - r_2}, \quad
    a_3 = \frac{a}{r_1 - r_2},
\end{equation*}
with $a = 0.1$.

The adiabatic index $\gamma$ is set to 2, and computations are performed up to a final time of $T = 1.0$. Periodic boundary conditions are applied along each direction. Since the solution is stationary, the experimental order of convergence (EOC) is computed using the $L^1$ norm errors in $\rho$, $u$, $v$, and $\theta$, evaluated against the initial data for various values of $\varepsilon$. The results are presented in Table~\ref{tab:ept_vortex_eoc_l-one}. The EOC values confirm uniform first-order convergence of the scheme, independent of the value of $\veps$. 

The stationary vortex is a benchmark for assessing the numerical dissipation of a scheme and its dependence on $\veps$. To this end, we compute the flow Mach number, defined as
$$
M = \sqrt{\frac{u^2 + v^2}{\gamma p / \rho}},  
$$
along with the relative kinetic energy over time for $\veps = 10^{-1}, 10^{-2}, 10^{-3}, 10^{-4}$ on a $100 \times 100$ grid. Figure~\ref{fig:stationary_vort_mach-num_pcol} presents pseudocolour plots of the Mach number ratio for various values of $\varepsilon$. A comparison with the initial Mach number ratio profile shown in Figure~\ref{fig:stationary_vort_mach-num_init} suggests that the scheme’s dissipation is effectively independent of $\varepsilon$. This observation is further supported by the relative kinetic energy plots in Figure~\ref{fig:stationary_vort_ke}, which demonstrate that the kinetic energy evolution over time remains nearly identical across all tested values of $\varepsilon$.

Finally, to further validate the asymptotic convergence result, we determine the rate of convergence of $\rho^\veps\theta^\veps$ to $1$ as $\veps\to 0$. In Figure~\ref{fig:stationary_vort_rho-theta}, we present the plots of  $\| \rho^{\veps}\theta^\veps(t, \cdot) - 1 \|_{L^\gm(\Omega)}$ over time for various values of $\veps$, and a plot of $\| \rho^{\veps} \theta^\veps - 1 \|_{L^\infty(0,T;L^\gm(\Omega))}$ with respect to $\veps$. The left panel of Figure~\ref{fig:stationary_vort_rho-theta} clearly shows that $\| \rho^{\veps}\theta^\veps(t, \cdot) - 1 \|_{L^\gm(\Omega)}$ decreases as $\veps \to 0$, and the right panel shows that the decay is of $\mcal{O}(\veps^2)$. In Table~\ref{tab:ept_vortex_convergence_l-infty}, the scheme's asymptotic consistency is again verified by evaluating the error—measured in the rigorous norm defined in Section~\ref{sec:ept_cont-case}—between the numerical solution and the limiting system solution in the low Mach number regime, on a fixed grid of size $100 \times 100$.

\begin{table}[ht]
\centering
\begin{tabular}{cccccccccc}
\toprule
$\veps$ & $N$ 
& $L^1$ error in $\rho$ & EOC 
& $L^1$ error in $u$ & EOC 
& $L^1$ error in $v$ & EOC 
& $L^1$ error in $\theta$ & EOC \\
\midrule
\multirow{4}{*}{$10^{-1}$} 
& 50  & 0.005039 & --    & 0.001610 & --    & 0.001610 & --    & 0.005078 & --    \\
& 100 & 0.002714 & 0.89  & 0.000942 & 0.77  & 0.000942 & 0.77  & 0.002724 & 0.90  \\
& 200 & 0.001413 & 0.94  & 0.000545 & 0.78  & 0.000545 & 0.78  & 0.001415 & 0.94  \\
& 400 & 0.000647 & 1.13  & 0.000314 & 0.80  & 0.000314 & 0.80  & 0.000647 & 1.13  \\
\midrule
\multirow{4}{*}{$10^{-2}$} 
& 50  & 0.005040 & --    & 0.001610 & --    & 0.001610 & --    & 0.005077 & --    \\
& 100 & 0.002714 & 0.89  & 0.000942 & 0.77  & 0.000942 & 0.77  & 0.002723 & 0.90  \\
& 200 & 0.001413 & 0.94  & 0.000545 & 0.78  & 0.000545 & 0.78  & 0.001415 & 0.94  \\
& 400 & 0.000647 & 1.13  & 0.000314 & 0.80  & 0.000314 & 0.80  & 0.000647 & 1.13  \\
\midrule
\multirow{4}{*}{$10^{-3}$} 
& 50  & 0.005040 & --    & 0.001610 & --    & 0.001610 & --    & 0.005078 & --    \\
& 100 & 0.002714 & 0.89  & 0.000942 & 0.77  & 0.000942 & 0.77  & 0.002724 & 0.90  \\
& 200 & 0.001413 & 0.94  & 0.000545 & 0.78  & 0.000545 & 0.78  & 0.001415 & 0.94  \\
& 400 & 0.000647 & 1.13  & 0.000314 & 0.80  & 0.000314 & 0.80  & 0.000647 & 1.13  \\
\midrule
\multirow{4}{*}{$10^{-4}$} 
& 50  & 0.005039 & --    & 0.001610 & --    & 0.001610 & --    & 0.005077 & --    \\
& 100 & 0.002714 & 0.89  & 0.000942 & 0.77  & 0.000942 & 0.77  & 0.002724 & 0.90  \\
& 200 & 0.001413 & 0.94  & 0.000545 & 0.78  & 0.000545 & 0.78  & 0.001415 & 0.94  \\
& 400 & 0.000647 & 1.13  & 0.000314 & 0.80  & 0.000314 & 0.80  & 0.000647 & 1.13  \\
\bottomrule
\end{tabular}
\caption{$L^1$ norm errors and experimental order of convergence (EOC) for the density $\rho$, velocities $u$, $v$, and temperature $\theta$ for different values of $\veps$.}
\label{tab:ept_vortex_eoc_l-one}
\end{table}

\begin{table}[ht]
\centering
\begin{tabular}{|c|c|c|c|c|c|}
\hline
$\veps\ \rule{0pt}{2.5ex}$ & $10^{-1}$ & $10^{-2}$ & $10^{-3}$ & $10^{-4}$ \\
\hline
$\left\| \rho^{\veps} - \tilde\rho \right\|_{L^{\infty}(0,T;L^{\gm}(\Omega))}$
 & 2.9009e-4 & 2.8157e-4 & 2.8150e-4 & 2.8148e-4 \\
 \hline
$\left\| \uu{u}^{\veps} - \uu{U}\right\|_{L^{2}(0,T;L^{2}(\Omega)^2)}$  & 8.1806e-4 & 8.1804e-4 & 8.1804e-4 & 8.1351e-4 \\
 \hline
$\left\| \theta^{\veps} - \tilde\theta \right\|_{L^{\infty}(0,T;L^{\gm}(\Omega))}$ & 2.8162e-4 & 2.8161e-4 & 2.8160e-4 & 2.8160e-4 \\
\hline
\end{tabular}
\caption{Error between numerical solution and limiting solution in the zero Mach number at $T=1.0$}
\label{tab:ept_vortex_convergence_l-infty}
\end{table}

\begin{figure}[htbp]
  \centering
  \includegraphics[height=0.25\textheight]{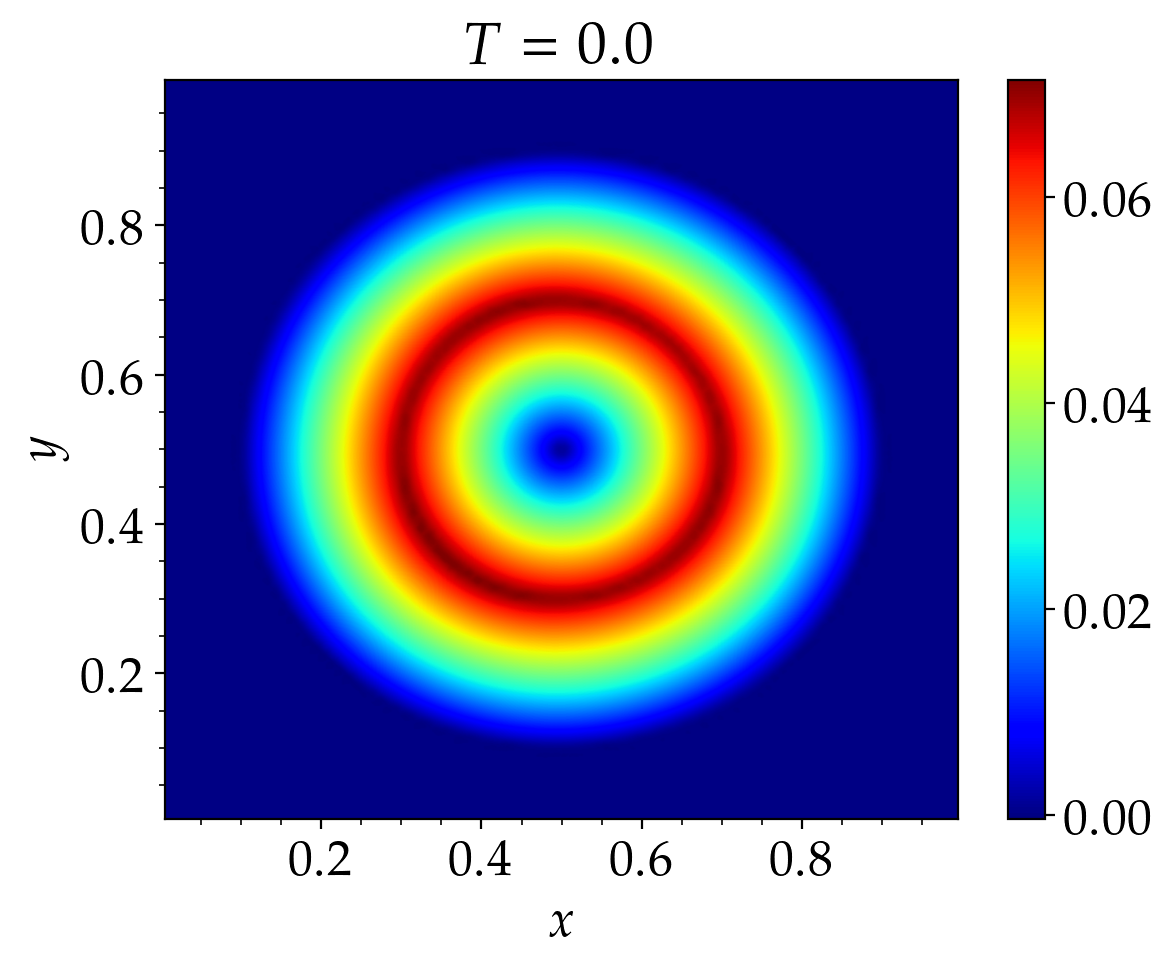}
  \caption{Pseudocolour plot of the initial Mach number ratio.} 
  \label{fig:stationary_vort_mach-num_init}
\end{figure}

\begin{figure}[htbp]
  \centering
    \includegraphics[height=0.5\textheight]{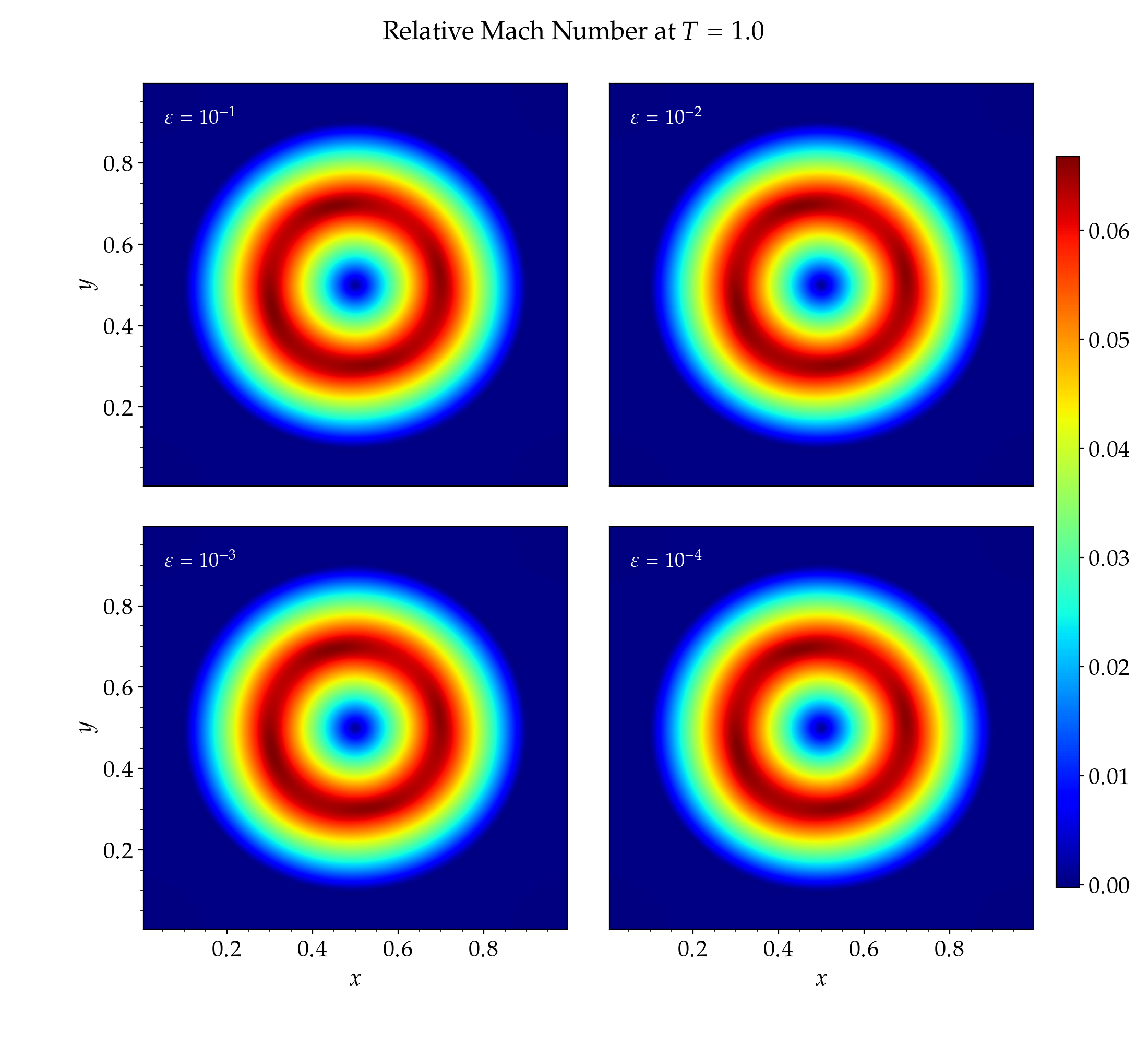}
  \caption{Pseudocolour plots of the Mach number ratio at time $T=1.0$ for
    $\veps = 10^{-1}, 10^{-2}, 10^{-3}, 10^{-4}$.} 
    \label{fig:stationary_vort_mach-num_pcol}
\end{figure}

\begin{figure}[htbp]
  \centering
  \includegraphics[height=0.225\textheight]{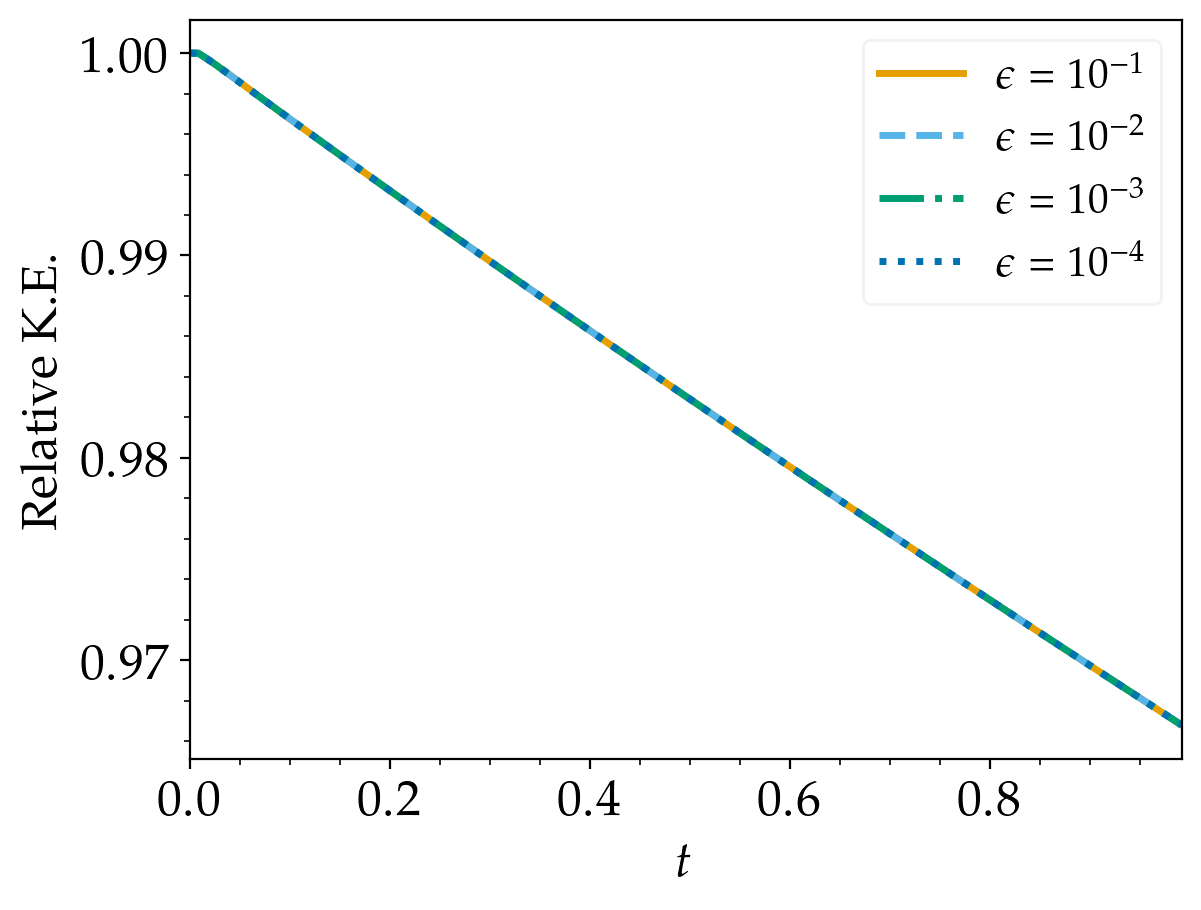}
  \caption{Relative kinetic energy for $\veps = 10^{-1}, 10^{-2}, 10^{-3}, 10^{-4}$.}   
  \label{fig:stationary_vort_ke}
\end{figure}

\begin{figure}[htbp]
  \centering
  \includegraphics[height=0.23\textheight]{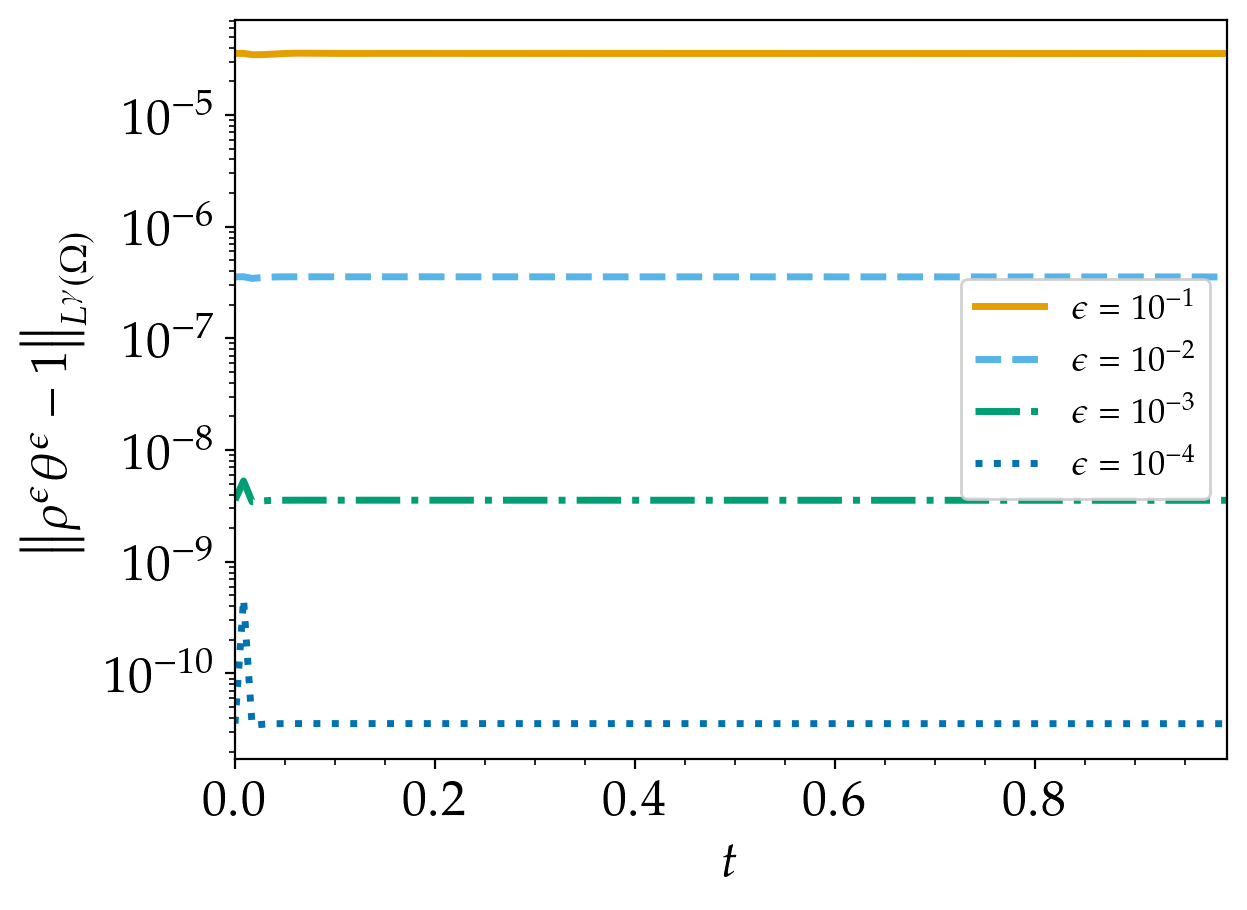}
  \includegraphics[height=0.23\textheight]{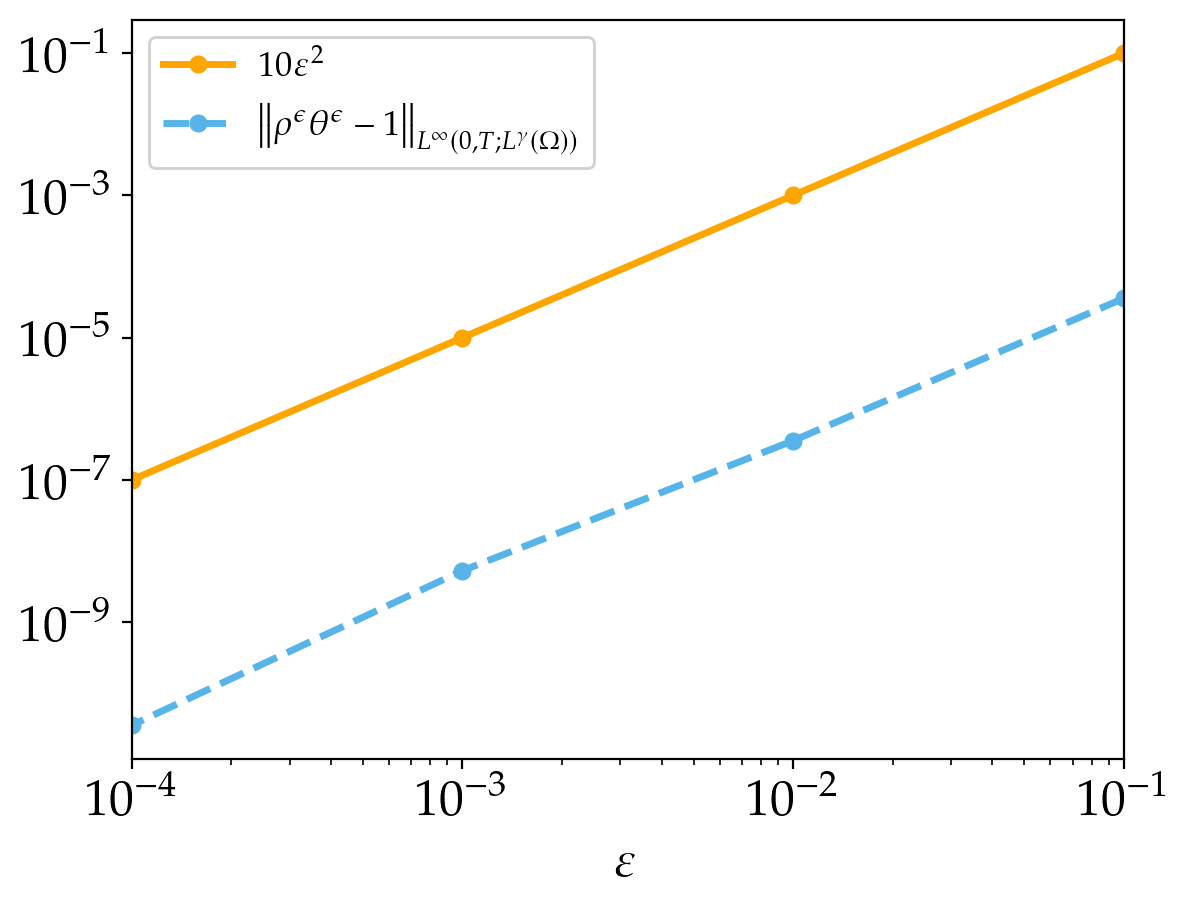}
  \caption{The value of $\norm{\rho^{\veps}\theta^{\veps}-1}_{L^{\gm}(\Omega)}$ with respect to time for different $\veps$ (left) and $\norm{\rho^{\veps}\theta^{\veps}-1}_{L^{\infty}(0,T;L^{\gm}(\Omega))}$ with respect to $\veps$ (right).}   
  \label{fig:stationary_vort_rho-theta}
\end{figure}

\subsection{Cylindrical Explosion}
\label{subsec:ept_cyl_expl}
We now consider a 2d cylindrical explosion problem, inspired by \cite{DLV17}. This test case serves two main purposes: first, to demonstrate the scheme’s ability to capture shock waves in the compressible regime; and second, to evaluate its performance in the incompressible limit as $\varepsilon \to 0$. To this end, we define the initial density as
\begin{equation*}
    \rho(0,x,y) = 
    \begin{cases}
        1 + \veps^2, & \mbox{if} \ r\leqs\frac{1}{2},
        \\
        1, & \mbox{otherwise},
    \end{cases}
\end{equation*}
where $r=\sqrt{x^2+y^2}$. The initial velocity field is taken as
\begin{equation*}
    (u,v)(0,x,y) = -\frac{\alpha(x,y)}{\rho(0,x,y)}\Big(\frac{x}{r},\frac{y}{r}\Big)\mcal{X}_{r>10^{-15}},
\end{equation*}
where $\alpha(x,y)=\max\{0,1-r\}(1-e^{-16r^2})$, and the initial temperature is set to $\theta(0,x,y)=1$. The computational domain $[-1,1]\times [-1,1]$ is uniformly divided into a $200\times200$ grid, and the boundaries are periodic everywhere. The adiabatic index is set to $\gm=1$. The problem is simulated in the compressible regime by setting $\veps=1$ and in the incompressible regime by setting a very small value, $\veps = 10^{-4}$.

In the compressible regime, we run simulations at various time instances and present the corresponding pseudocolour plots in Figure~\ref{fig:ept_cyl_exp_eps1p0_soln}. The results display a circular shockwave propagating outward over time, demonstrating the scheme’s capability to accurately capture shock dynamics.
\begin{figure}[htbp]
    \centering
    \includegraphics[height=0.65\textheight]{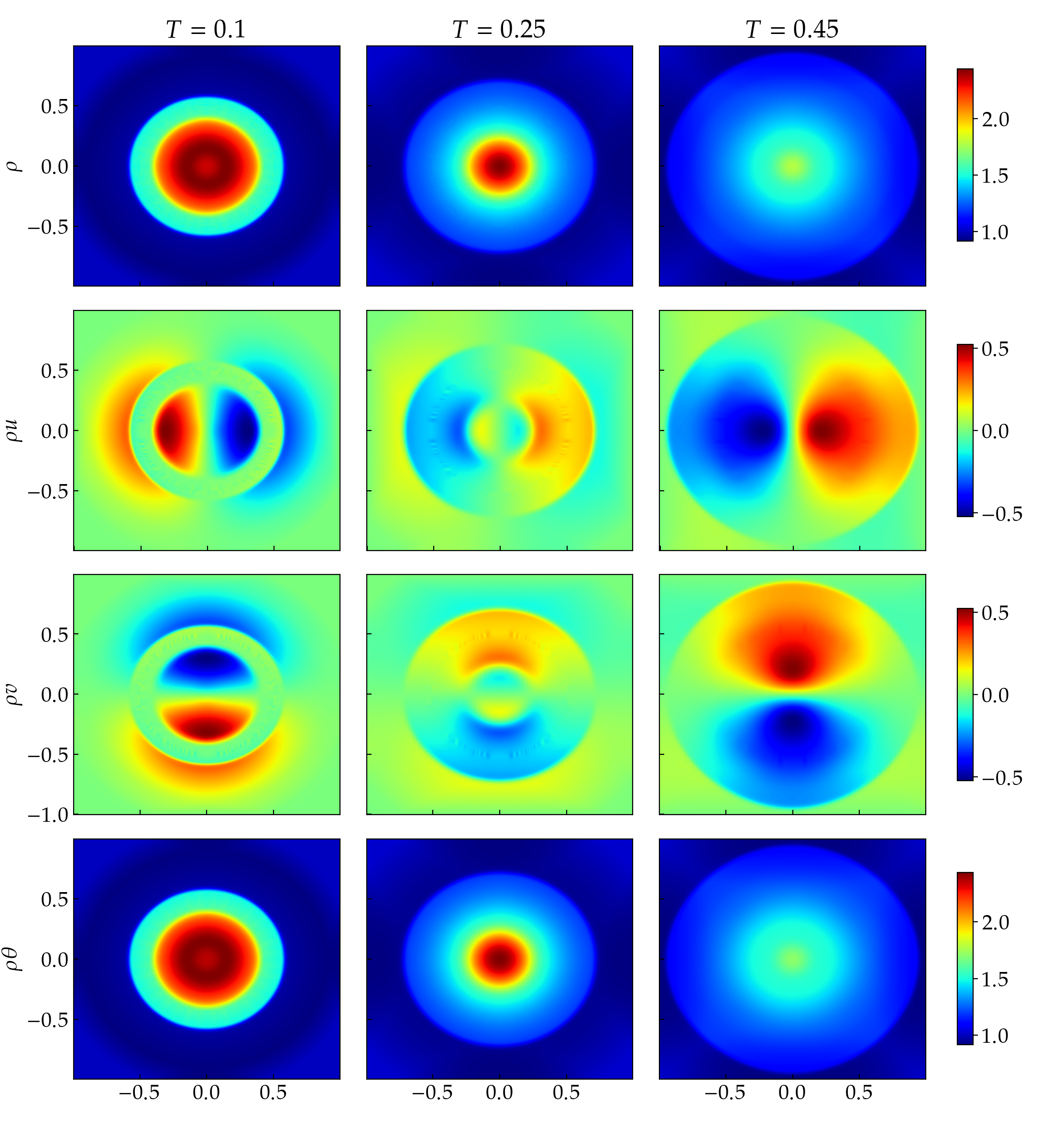}
    \caption{Pseudocolour plot of the density, the momentum and the total potential temperature for the cylindrical explosion problem at different times for $\veps=1$.}
    \label{fig:ept_cyl_exp_eps1p0_soln}
\end{figure}
In the incompressible regime, the simulation is carried out up to a final time of $T = 0.05$. Figure~\ref{fig:cyl_exp_eps-04_den_div_vel} presents surface plots of the deviation of the total potential temperature from unity and the divergence of the velocity field. As expected, the total potential temperature remains nearly constant at $\Theta = 1$, with deviations on the order of $10^{-9}$, i.e., $\mathcal{O}(\varepsilon^2)$. Additionally, the velocity divergence remains small, on the order of $\mathcal{O}(10^{-4})$.
\begin{figure}[htbp]
    \centering 
    \includegraphics[height=0.22\textheight]{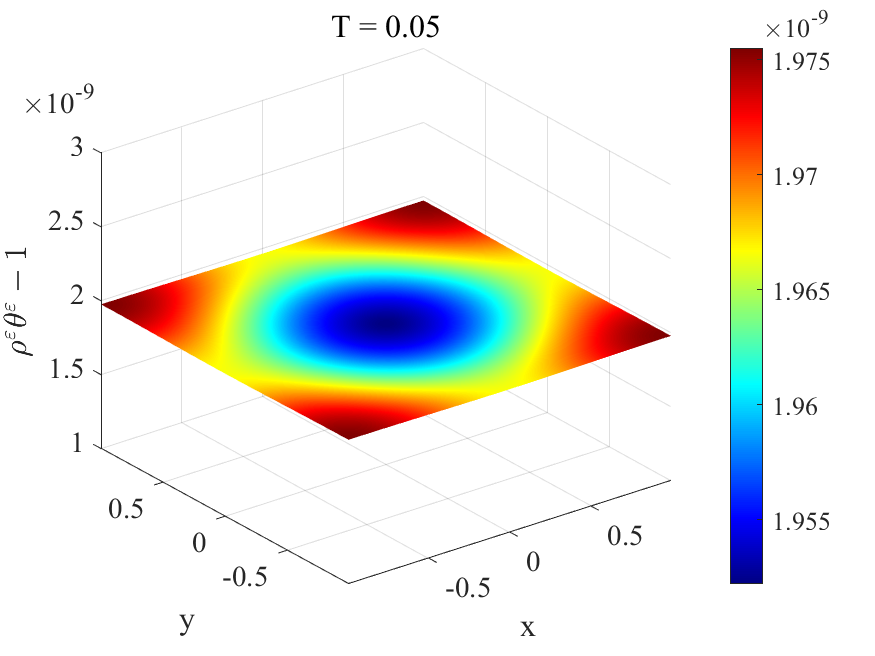}
    \includegraphics[height=0.22\textheight]{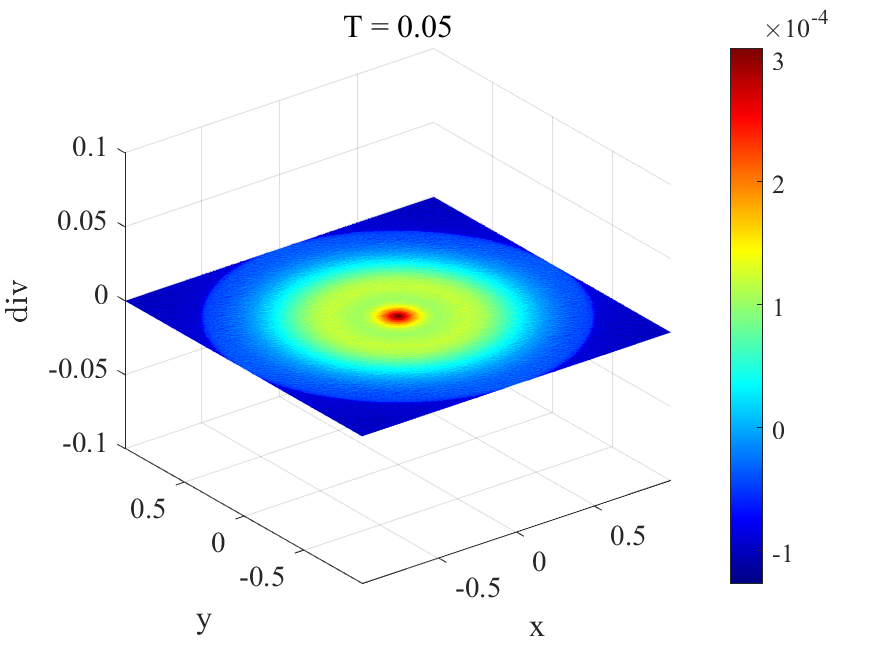}
    \caption{Surface plots of the deviation of the total potential temperature from $1$ and the divergence of the velocity for the cylindrical explosion problem at time $T=0.05$ for $\veps=10^{-4}$.}
    \label{fig:cyl_exp_eps-04_den_div_vel}
\end{figure}

\subsection{Baroclinic Vorticity Generation}
\label{subsec:ept_baroclinic_vorticity}
This test problem, taken from \cite{NBA+14}, involves the two-dimensional interaction between an acoustic wave and a stratified density profile. Since the limiting system \eqref{eq:incomp_den-dep_sys} permits density variations, the objective of this problem is to demonstrate the scheme’s ability to resolve complex density structures in the low Mach number regime. The initial conditions read
\begin{equation}
\begin{aligned}
\rho(0, x, y) &= 1 + \frac{\veps}{2000} \left(1 + \cos(\pi x)\right) + \Phi(y), \\
u(0, x, y) &= \frac{1}{2}\sqrt{\gamma}(1 + \cos(\pi x)), \ v(0, x, y) = 0, \\
(\rho\theta)^\gamma(0, x, y) &= 1 + \frac{1}{2} \veps \gm (1 + \cos(\pi x)),
\end{aligned}
\end{equation}
where we set $\gm = 1.4$. The function $\Phi$ is defined as
\begin{equation*}
    \Phi(y) =
\begin{cases}
4.5y, & \text{if } 0 \leqs y \leqs \dfrac{1}{5}, \\
4.5 y - 1.8, & \text{otherwise}.
\end{cases}
\end{equation*}
We consider a weakly compressible regime with $\varepsilon = 0.05$, using a computational domain of $[-1, 1] \times [0, 0.4]$ and periodic boundary conditions in both directions. The domain is discretised using an $800 \times 160$ grid. Figure~\ref{fig:ept_bar-vort_solnDN} and Figure~\ref{fig:ept_bar-vort_solnTH} illustrate the time evolution of the density and potential temperature, respectively. Both figures reveal two distinct layers in the initial distributions, separated by a narrow vertical fluctuation. Due to differing accelerations induced by the acoustic perturbation, rotational motion develops along the interface, leading to the formation of a long-wavelength sinusoidal shear layer by time $T = 0.5$. After several passes of the acoustic wave, this shear layer becomes unstable. By $T = 1$, the instabilities have amplified significantly, resulting in the formation of Kelvin–Helmholtz vortices. Furthermore, since $\rho^\veps\theta^\veps\to 1$, the values of $\rho^\veps$ and $\theta^\veps$ are inversely proportional.   
\begin{figure}[htbp]
    \centering
    \includegraphics[height=0.35\textheight]{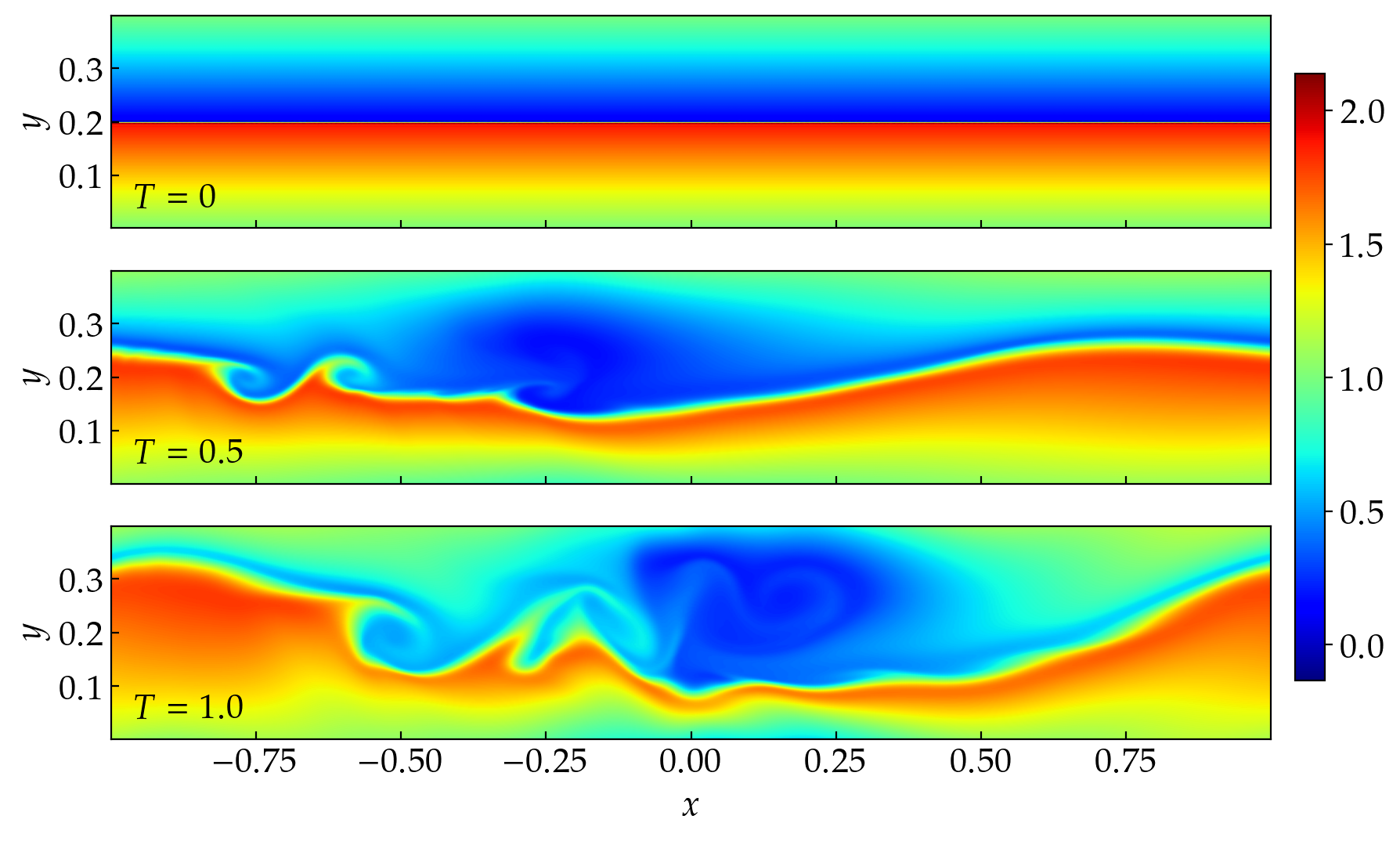}
    \caption{Pseudocolour plot of the density at different times for the baroclinic vorticity generation problem.}
    \label{fig:ept_bar-vort_solnDN}
\end{figure}
\begin{figure}[htbp]
    \centering
    \includegraphics[height=0.35\textheight]{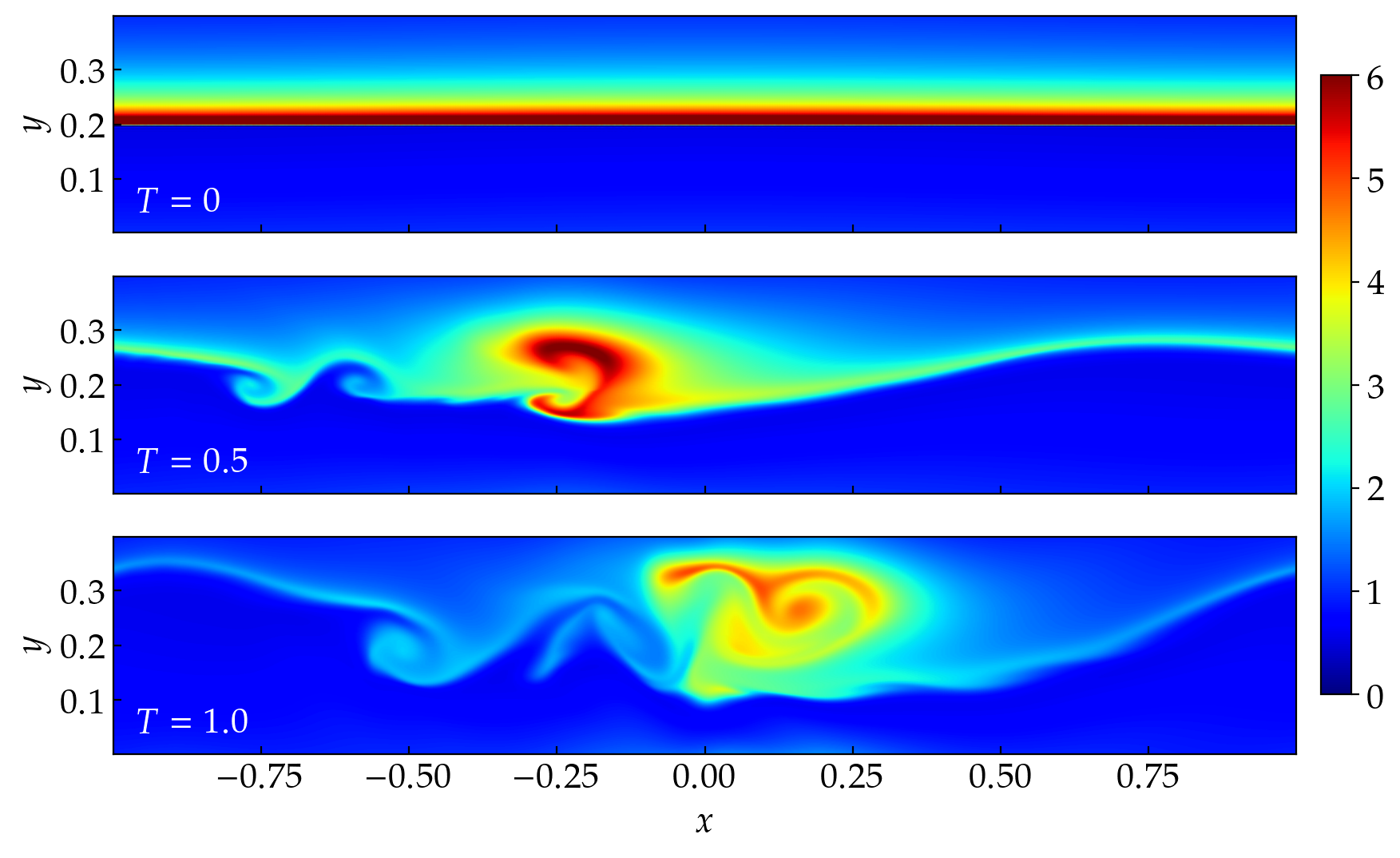}
    \caption{Pseudocolour plot of the potential temperature at different times for the baroclinic vorticity generation problem.}
    \label{fig:ept_bar-vort_solnTH}
\end{figure}

\section{Concluding Remarks}
\label{sec:ept_conclusion}
We have developed a semi-implicit, all-speed, structure-preserving finite volume scheme for the compressible Euler equations with potential temperature transport. The proposed method is asymptotic preserving, enabling stable and accurate simulations across a broad range of Mach numbers, including the challenging low Mach number regime. We proved consistency with both the compressible system and its incompressible density-dependent limit, providing a rigorous foundation for the scheme’s applicability in multiscale problems. Numerical experiments confirm the robustness of the method and its ability to capture essential flow features in both compressible and nearly incompressible regimes. This work contributes toward reliable and efficient numerical modelling of adiabatic atmospheric and astrophysical flows, where both compressibility effects and low Mach number behaviour coexist.

\bibliographystyle{abbrv}
\bibliography{references}

\begin{thebibliography}{10}

\bibitem{AGK23}
K.~R. Arun, R.~Ghorai, and M.~Kar.
\newblock An {A}symptotic {P}reserving and {E}nergy {S}table {S}cheme for the {B}arotropic {E}uler {S}ystem in the {I}ncompressible {L}imit.
\newblock {\em J. Sci. Comput.}, 97(3):Paper No. 73, 2023.

\bibitem{AS20}
K.~R. Arun and S.~Samantaray.
\newblock Asymptotic preserving low {M}ach number accurate {IMEX} finite volume schemes for the isentropic {E}uler equations.
\newblock {\em J. Sci. Comput.}, 82(2):Art. 35, 32, 2020.

\bibitem{BW98}
H.~Bijl and P.~Wesseling.
\newblock A unified method for computing incompressible and compressible flows in boundary-fitted coordinates.
\newblock {\em J. Comput. Phys.}, 141(2):153--173, 1998.

\bibitem{BAL+14}
G.~Bispen, K.~R. Arun, M.~Luk\'{a}\v{c}ov\'{a}-Medvi\v{d}ov\'{a}, and S.~Noelle.
\newblock I{MEX} large time step finite volume methods for low {F}roude number shallow water flows.
\newblock {\em Commun. Comput. Phys.}, 16(2):307--347, 2014.

\bibitem{BJR+19}
S.~Boscarino, J.-M. Qiu, G.~Russo, and T.~Xiong.
\newblock A high order semi-implicit {IMEX} {WENO} scheme for the all-{M}ach isentropic {E}uler system.
\newblock {\em J. Comput. Phys.}, 392:594--618, 2019.

\bibitem{CGK16}
C.~Chalons, M.~Girardin, and S.~Kokh.
\newblock An all-regime {L}agrange-projection like scheme for the gas dynamics equations on unstructured meshes.
\newblock {\em Commun. Comput. Phys.}, 20(1):188--233, 2016.

\bibitem{CP99}
P.~Colella and K.~Pao.
\newblock A projection method for low speed flows.
\newblock {\em J. Comput. Phys.}, 149(2):245--269, 1999.

\bibitem{CDV17}
F.~Couderc, A.~Duran, and J.-P. Vila.
\newblock An explicit asymptotic preserving low {F}roude scheme for the multilayer shallow water model with density stratification.
\newblock {\em J. Comput. Phys.}, 343:235--270, 2017.

\bibitem{DT11}
P.~Degond and M.~Tang.
\newblock All speed scheme for the low {M}ach number limit of the isentropic {E}uler equations.
\newblock {\em Commun. Comput. Phys.}, 10(1):1--31, 2011.

\bibitem{Dei85}
K.~Deimling.
\newblock {\em Nonlinear functional analysis}.
\newblock Springer-Verlag, Berlin, 1985.

\bibitem{DLV17}
G.~Dimarco, R.~Loub\`ere, and M.-H. Vignal.
\newblock Study of a new asymptotic preserving scheme for the {E}uler system in the low {M}ach number limit.
\newblock {\em SIAM J. Sci. Comput.}, 39(5):A2099--A2128, 2017.

\bibitem{DVB17}
A.~Duran, J.-P. Vila, and R.~Baraille.
\newblock Semi-implicit staggered mesh scheme for the multi-layer shallow water system.
\newblock {\em C. R. Math. Acad. Sci. Paris}, 355(12):1298--1306, 2017.

\bibitem{DVB20}
A.~Duran, J.-P. Vila, and R.~Baraille.
\newblock Energy-stable staggered schemes for the {S}hallow {W}ater equations.
\newblock {\em J. Comput. Phys.}, 401:109051, 24, 2020.

\bibitem{FDK07}
M.~Feistauer, V.~Dolej{\v s}{\'i}, and V.~Ku{\v c}era.
\newblock On the discontinuous {G}alerkin method for the simulation of compressible flow with wide range of {M}ach numbers.
\newblock {\em Comput. Vis. Sci.}, 10(1):17--27, 2007.

\bibitem{GHL22}
T.~Gallou\"{e}t, R.~Herbin, and J.-C. Latch\'{e}.
\newblock Lax--{W}endroff consistency of finite volume schemes for systems of non linear conservation laws: extension to staggered schemes.
\newblock {\em SeMA J.}, 79(2):333--354, 2022.

\bibitem{GMN19}
T.~Gallou\"{e}t, D.~Maltese, and A.~Novotny.
\newblock Error estimates for the implicit {MAC} scheme for the compressible {N}avier-{S}tokes equations.
\newblock {\em Numer. Math.}, 141(2):495--567, 2019.

\bibitem{GVV13}
N.~Grenier, J.-P. Vila, and P.~Villedieu.
\newblock An accurate low-{M}ach scheme for a compressible two-fluid model applied to free-surface flows.
\newblock {\em J. Comput. Phys.}, 252:1--19, 2013.

\bibitem{HJL12}
J.~Haack, S.~Jin, and J.-G. Liu.
\newblock An all-speed asymptotic-preserving method for the isentropic {E}uler and {N}avier-{S}tokes equations.
\newblock {\em Commun. Comput. Phys.}, 12(4):955--980, 2012.

\bibitem{HA68}
F.~H. Harlow and A.~A. Amsden.
\newblock Numerical calculation of almost incompressible flow.
\newblock {\em J. Comput. Phys.}, 3(1):80--93, 1968.

\bibitem{HW65}
F.~H. Harlow and J.~E. Welch.
\newblock Numerical calculation of time-dependent viscous incompressible flow of fluid with free surface.
\newblock {\em Phys. Fluids}, 8(12):2182--2189, 1965.

\bibitem{HLN+23}
R.~Herbin, J.-C. Latch\'{e}, Y.~Nasseri, and N.~Therme.
\newblock A consistent quasi-second-order staggered scheme for the two-dimensional shallow water equations.
\newblock {\em IMA J. Numer. Anal.}, 43(1):99--143, 2023.

\bibitem{HLS21}
R.~Herbin, J.-C. Latch\'{e}, and K.~Saleh.
\newblock Low {M}ach number limit of some staggered schemes for compressible barotropic flows.
\newblock {\em Math. Comp.}, 90(329):1039--1087, 2021.

\bibitem{IGW86}
R.~I. Issa, A.~D. Gosman, and A.~P. Watkins.
\newblock The computation of compressible and incompressible recirculating flows by a noniterative implicit scheme.
\newblock {\em J. Comput. Phys.}, 62(1):66--82, 1986.

\bibitem{Jin99}
S.~Jin.
\newblock Efficient asymptotic-preserving {(AP)} schemes for some multiscale kinetic equations.
\newblock {\em SIAM J. Sci. Comput.}, 21(2):441--454, 1999.

\bibitem{KP89}
K.~C. Karki and S.~V. Patankar.
\newblock Pressure based calculation procedure for viscous flows at all speedsin arbitrary configurations.
\newblock {\em AIAA Journal}, 27(9):1167--1174, 1989.

\bibitem{KM82}
S.~Klainerman and A.~Majda.
\newblock Compressible and incompressible fluids.
\newblock {\em Comm. Pure Appl. Math.}, 35(5):629--651, 1982.

\bibitem{Kle95}
R.~Klein.
\newblock Semi-implicit extension of a {G}odunov-type scheme based on low {M}ach number asymptotics. {I}. {O}ne-dimensional flow.
\newblock {\em J. Comput. Phys.}, 121(2):213--237, 1995.

\bibitem{Kle10}
R.~Klein.
\newblock Scale-dependent models for atmospheric flows.
\newblock In {\em Annual review of fluid mechanics. {V}ol. 42}, volume~42 of {\em Annu. Rev. Fluid Mech.}, pages 249--274. Annual Reviews, Palo Alto, CA, 2010.

\bibitem{KSG+09}
N.~Kwatra, J.~Su, J.~T. Gr\'{e}tarsson, and R.~Fedkiw.
\newblock A method for avoiding the acoustic time step restriction in compressible flow.
\newblock {\em J. Comput. Phys.}, 228(11):4146--4161, 2009.

\bibitem{Lio96}
P.-L. Lions.
\newblock {\em Mathematical topics in fluid mechanics. {V}ol. 1}, volume~3 of {\em Oxford Lecture Series in Mathematics and its Applications}.
\newblock The Clarendon Press, Oxford University Press, New York, 1996.
\newblock Incompressible models, Oxford Science Publications.

\bibitem{Lio98}
P.-L. Lions.
\newblock {\em Mathematical topics in fluid mechanics. {V}ol. 2}, volume~10 of {\em Oxford Lecture Series in Mathematics and its Applications}.
\newblock The Clarendon Press, Oxford University Press, New York, 1998.
\newblock Compressible models, Oxford Science Publications.

\bibitem{LM98}
P.-L. Lions and N.~Masmoudi.
\newblock Incompressible limit for a viscous compressible fluid.
\newblock {\em J. Math. Pures Appl. (9)}, 77(6):585--627, 1998.

\bibitem{MKB+12}
Y.~Moguen, T.~Kousksou, P.~Bruel, J.~Vierendeels, and E.~Dick.
\newblock Pressure-velocity coupling allowing acoustic calculation in low {M}ach number flow.
\newblock {\em J. Comput. Phys.}, 231(16):5522--5541, 2012.

\bibitem{MRK+03}
C.-D. Munz, S.~Roller, R.~Klein, and K.~J. Geratz.
\newblock The extension of incompressible flow solvers to the weakly compressible regime.
\newblock {\em Comput. \& Fluids}, 32(2):173--196, 2003.

\bibitem{NIR01}
L.~Nirenberg.
\newblock {\em Topics in nonlinear functional analysis}, volume~6 of {\em Courant Lecture Notes in Mathematics}.
\newblock New York University, Courant Institute of Mathematical Sciences, New York; American Mathematical Society, Providence, RI, 2001.
\newblock Chapter 6 by E. Zehnder, Notes by R. A. Artino, Revised reprint of the 1974 original.

\bibitem{NBA+14}
S.~Noelle, G.~Bispen, K.~R. Arun, M.~Luk\'{a}\v{c}ov\'{a}-Medvi\v{d}ov\'{a}, and C.-D. Munz.
\newblock A weakly asymptotic preserving low {M}ach number scheme for the {E}uler equations of gas dynamics.
\newblock {\em SIAM J. Sci. Comput.}, 36(6):B989--B1024, 2014.

\bibitem{DYY06}
D.~O'Regan, Y.~J. Cho, and Y.-Q. Chen.
\newblock {\em Topological degree theory and applications}, volume~10 of {\em Series in Mathematical Analysis and Applications}.
\newblock Chapman \& Hall/CRC, Boca Raton, FL, 2006.

\bibitem{PV16}
M.~Parisot and J.-P. Vila.
\newblock Centered-potential regularization for the advection upstream splitting method.
\newblock {\em SIAM J. Numer. Anal.}, 54(5):3083--3104, 2016.

\bibitem{PM05}
J.~H. Park and C.-D. Munz.
\newblock Multiple pressure variables methods for fluid flow at all {M}ach numbers.
\newblock {\em Internat. J. Numer. Methods Fluids}, 49(8):905--931, 2005.

\bibitem{Sch94}
S.~Schochet.
\newblock Fast singular limits of hyperbolic {PDE}s.
\newblock {\em J. Differential Equations}, 114(2):476--512, 1994.

\bibitem{SKW14}
P.~K. Smolarkiewicz, C.~K{\"u}hnlein, and N.~P. Wedi.
\newblock A consistent framework for discrete integrations of soundproof and compressible {PDE}s of atmospheric dynamics.
\newblock {\em J. Comput. Phys.}, 263:185--205, 2014.

\bibitem{TPK20}
A.~Thomann, G.~Puppo, and C.~Klingenberg.
\newblock An all speed second order well-balanced {IMEX} relaxation scheme for the {E}uler equations with gravity.
\newblock {\em J. Comput. Phys.}, 420:109723, 25, 2020.

\bibitem{WPM02}
C.~Wall, C.~D. Pierce, and P.~Moin.
\newblock A semi-implicit method for resolution of acoustic waves in low {M}ach number flows.
\newblock {\em J. Comput. Phys.}, 181(2):545--563, 2002.

\bibitem{Zen18}
M.~Zenk.
\newblock {\em On Numerical Methods for Astrophysical Applications}.
\newblock Doctoral thesis, Universit{\"a}t W{\"u}rzburg, 2018.

\end{thebibliography}

\end{document}